\documentclass[makeidx,a4paper,11pt]{article}
\usepackage{makeidx}
\usepackage{amssymb , amsmath, amsthm}
\newtheorem{defi}{Definition} 
\newtheorem{thm}[defi]{Theorem}

 \newtheorem{prop}[defi]{Proposition}
\newtheorem{lemme}[defi]{Lemma}
\newtheorem{cor}[defi]{Corollary}

\newcommand{\tre}[3]{#1&#2&#3\\}
\newcommand{\quattro}[4]{#1&#2&#3&#4\\}
\newcommand{\cinque}[5]{#1&#2&#3&#4&#5\\}

\newcommand{\matrice}{\begin{pmatrix}}
\newcommand{\ok}{\end{pmatrix}}
\newcommand{\dmatrice}{\begin{vmatrix}}
\newcommand{\dok}{\end{vmatrix}}
 
\newcommand{\twosystem}[2]{\left\{\begin{aligned} &#1\\ &#2\end{aligned}\right.}
\newcommand{\threesystem}[3]{\left\{ \begin{aligned}&#1\\ &#2\\&#3\end{aligned}\right.}

\newcommand{\twovector}[2]{\begin{pmatrix}#1\\#2\end{pmatrix}}
\newcommand{\threevector}[3]{\begin{pmatrix} #1\\#2\\#3\end{pmatrix}}
\newcommand{\fourvector}[4]{\begin{pmatrix} #1\\#2\\#3\\#4\end{pmatrix}}

\newcommand{\twomatrix}[4]{\begin{pmatrix} #1&#2\\#3&#4\end{pmatrix}}
\newcommand{\threematrix}[9]{\begin{pmatrix} #1&#2&#3\\#4&#5&#6\\#7&#8&#9\end{pmatrix}}
\newcommand{\twodet}[4]{\begin{vmatrix} #1&#2\\#3&#4\end{vmatrix}}
\newcommand{\threedet}[9]{\begin{vmatrix} #1&#2&#3\\#4&#5&#6\\#7&#8&#9\end{vmatrix}}

\newcommand{\nero}{\smallskip$\bullet\quad$\rm}

\newcommand{\scal}[2]{\langle{#1},{#2}\rangle}

\newcommand{\abs}[1]{\lvert{#1}\rvert}

\newcommand{\reals}{{\bf R}}
\newcommand{\sphere}[1]{{\bf S}^{#1}}
\newcommand{\hyp}[1]{{\bf H}^{#1}}
\newcommand{\real}[1]{{\bf R}^{#1}}

\newcommand{\tr}{{\rm tr}}

\newcommand{\bd}{\partial}

\newcommand{\derive}[2]{\dfrac{\bd #1}{\bd#2}}

\newcommand{\deriven}[3]{\dfrac{\bd^{#1} #2}{\bd #3^{#1}}}

\begin{document}

\title{Heat flow, heat content and the isoparametric property
\footnote{Classification AMS $2000$: 58J50, 35P15 \newline 
Keywords: Heat equation, heat flow, overdetermined problems, isoparametric hypersurfaces}}

\author{Alessandro Savo}
\date{}

\maketitle
\begin{abstract} 
Let $M$ be a Riemannian manifold and $\Omega$ a compact domain of $M$ with smooth boundary. We study the solution of the heat equation on $\Omega$ having constant unit initial conditions and Dirichlet boundary conditions. The purpose of this paper is to study the geometry of domains for which, at any fixed value of time,  the normal derivative of the solution (heat flow) is a constant function on the boundary. We express this fact by saying that such domains have the {\it constant flow property}. In constant curvature spaces known examples of such domains are given by geodesic balls  and, more generally, by domains whose boundary is connected and isoparametric. The question is: are they all like that?

In this paper we give an  affirmative answer to this question: in fact we prove more generally that, if a domain in an analytic Riemannian manifold has the constant flow property, then every component of its boundary is an isoparametric hypersurface.  For space forms, we also relate the order of vanishing of the heat content with fixed boundary data with the constancy of the $r$-mean curvatures of the boundary and with the isoparametric property. Finally, we discuss the constant flow property in relation to other well-known overdetermined problems involving the Laplace operator, like the Serrin problem or the Schiffer problem.

\end{abstract}
\large

\section{Introduction} 

In this paper, we prove a rigidity result for Riemannian manifolds with boundary satisfying a certain overdetermined problem for the heat equation; the aim is to understand the conditions on the heat content and the heat flow which insure the isoparametric property of the boundary. In this introduction, we first state the main results  in Section \ref{cfp} (Theorem \ref{main} and Theorem \ref{hciso});  then, in Section \ref{op}, we relate the constant flow property to other overdetermined problems.
The isoparametric property is recalled in Section \ref{ib} and, in Section \ref{gen}, we state a general result valid in any smooth Riemannian manifold (Theorem \ref{general}). In Section \ref{pmain}, we show how Theorem \ref{main} follows from Theorem \ref{general} and finally, in Section \ref{pop}, we give the plan of the paper with a rough scheme of the proofs.

\subsection{The constant flow property and the main results} \label{cfp}

 Let $M$ be a Riemannian manifold of dimension $n$, with metric tensor $g$, and let $\Omega$ be a compact domain in $M$  having smooth boundary $\bd\Omega$.  A basic object in heat diffusion 
is the solution $u(t,x)$ of the heat equation on $\Omega$ with initial data $1$ and Dirichlet boundary conditions:
\begin{equation}\label{temp}
\threesystem
{\Delta u+\derive ut=0\quad\text{on}\quad \Omega,}
{u(0,x)=1\quad\text{for all}\quad x\in\Omega,}
{u(t,y)=0\quad\text{for all $y\in\bd\Omega$ \, and $t>0$},}
\end{equation}
where $\Delta$ is the Laplace-Beltrami operator defined by the Riemannian metric $g$. We will often write $u(t,x)$ as  $u_t(x)$ so that $u_0=1$. The interest in the function $u$ is also given by the fact that
\begin{equation}\label{hk}
u(t,x)=\int_{\Omega}k(t,x,y)dy,
\end{equation}
where $k: (0,\infty)\times\Omega\times\Omega\to\reals $ is the heat kernel of $\Omega$ (that is, the fundamental solution of the heat equation with Dirichlet boundary conditions). Note that $u(t,x)$ is the temperature at time $t$, at the point $x\in\Omega$, assuming that the initial temperature distribution is constant, equal to $1$, and that the boundary $\bd\Omega$ is kept at temperature zero at all times.

Now let $\nu$ be the unit normal vector field of $\bd\Omega$, pointing inward, and let $y\in\bd\Omega$.  Then, $\derive{u}{\nu}(t,y)$ can be interpreted as the {\it heat flow} at time $t$, at the boundary point $y$.  

\medskip

\begin{defi}\label{chf} We say that $\Omega$ has the {\rm constant flow property} if, for all fixed  $t>0$, the heat flow
$$
\derive{u}{\nu}(t,\cdot):\bd\Omega\to\reals
$$ 
is a constant function on $\bd\Omega$.
\end{defi}

This property could be seen as an overdetermined problem for the heat equation. Overdetermined problems for the Laplacian have been vastly studied in the literature (see for example \cite{Ber},\cite{B-S},\cite{B-Y},\cite{ES-I},\cite {K-N},\cite{K-P},\cite{Liu},\cite{Mol},\cite {P-S},\cite{Ser},\cite {Wei},\cite{Wil}), and some of them will be recalled in Section \ref{op} below.  Solutions to a specific overdetermined problem exist only for special geometries, and in general one would like to classify  all domains which support such solutions. For the constant flow property above we have the following rigidity theorem. 

\begin{thm}\label{main} Let $\Omega$ be a compact domain with smooth boundary in an analytic Riemannian manifold $M$. Assume that it has the constant flow property. Then each component of  $\bd\Omega$ is an isoparametric hypersurface of $M$.
\end{thm}

Theorem \ref{main} follows from a more general result, valid  in arbitrary smooth Riemannian manifolds (Theorem \ref{general} below) and proved by studying the complete asymptotic expansion of the heat flow for small time, which was obtained in \cite{Savo1} and \cite{Savo2}.

The definition and the main properties of isoparametric hypersurfaces will be recalled in Section \ref{ib} below; we only recall here that, in space forms, isoparametric hypersurfaces are characterized by having constant principal curvatures. When the ambient space is $\real n$ or $\hyp n$  the only compact isoparametric hypersurfaces are the geodesic spheres, and Theorem \ref{main} asserts that in those cases the only domains with the constant flow property are geodesic balls (this is an easy case, immediately obtained from the Alexandrov theorem and the first two terms of the heat flow asymptotics, see the end of Section \ref{two} for the short proof).

Things get much more interesting and complicated  in the sphere $\sphere n$, also due to the fact that there is no analogue of the Alexandrov theorem, and that there is abundance of isoparametric hypersurfaces not isometric to geodesic spheres.  From the work of M\"unzner \cite{Mun} we know that a connected isoparametric hypersurface $\Sigma$ of the canonical sphere divides the sphere into two domains with common boundary $\Sigma$. 
It was proved by Shklover in (\cite{Shk}, Section 5.3 p. 562), that any spherical domain  with connected, isoparametric boundary has the constant flow property. 
This fact, together with Theorem \ref{main} above, gives the following characterization of the isoparametric property in space forms.

\begin{cor} \label{char} A compact, connected, hypersurface of a space form is isoparametric if and only if it bounds a domain having the constant flow property. 
\end{cor} 

Next, we examine the relations between the heat content and the isoparametric property. First, let us fix a smooth function $f\in C^{\infty}(\bd\Omega)$ and consider the solution $\tilde f_t(x)\doteq \tilde f(t,x)$ of the heat equation with zero initial temperature and with boundary temperature prescribed by the function $f$, that is:
\begin{equation}\label{pbc}
\threesystem
{\Delta \tilde f_t+\derive {\tilde f_t}t=0}
{\tilde f_0(x)=0\quad\text{for all}\quad x\in\Omega}
{\tilde f_t(y)=f(y)\quad\text{for all $y\in\bd\Omega$ \, and $t>0$}.}
\end{equation}
We  call the function of $t\in (0,\infty)$:
\begin{equation}\label{hcb}
\tilde H_{f}(t)\doteq \int_{\Omega}\tilde f_t
\end{equation}
the {\it heat content with boundary data $f$}. It is clear that $\lim_{t\to 0}\tilde H_{f}(t)=0$; if moreover $f$ integrates to zero on $\bd\Omega$ then $\tilde H_f(t)$ vanishes to order at least $1$ as $t\to 0$. In fact, it is easy to prove (see Theorem \ref{zeroflow}) that 

\nero {\it A domain $\Omega$ has the constant flow property if and only if the heat content with boundary data in $C^{\infty}_0(\bd\Omega)\doteq\{f\in C^{\infty}(\bd\Omega):\int_{\bd\Omega}f=0\}$ is equal to zero at all times.}

\medskip

Then, domains with the constant flow property are {\it perfect heat diffusers}: the incoming heat, flowing inside from theregions where the boundary temperature is positive is perfectly balanced, at each time, by the outgoing heat, flowing  away in correspondence to the negative boundary temperature: total heat content \eqref{hcb} is constant in time, hence always zero. This holds regardless of the temperature distribution on the boundary, as long as it has zero mean.

With this in mind, Theorem \ref{main} asserts that
if the heat flow with boundary data in $C^{\infty}_0(\bd\Omega)$ vanishes identically then each component of the boundary is isoparametric; that is,  perfect heat diffusers have isoparametric boundary (and viceversa, at least in space forms, if the boundary is connected).

\smallskip

In our second main result, we observe that the order of vanishing of the heat content with  boundary data in $C^{\infty}_0(\bd\Omega)$ is somewhat related to a kind of {\it degree of isoparametricity} of the boundary, at least  if the ambient space has constant curvature. Moreover, we will also prove that only a certain (finite) order of vanishing  is needed in order  to insure the isoparametric property: in a certain sense we can weaken, in space forms, the constant flow assumption of Theorem \ref{main}.

To be more precise, let $S$ denote the shape operator $\bd\Omega$, with eigenvalues (i.e. principal curvatures) denoted $k_1,\dots,k_{n-1}$. 
For $r=1,\dots,n-1$, define the {\it mean curvature of order $r$} as the $r$-th elementary function of the principal curvatures:
$$
E_r=\sum_{1\leq i_1<\dots<i_r\leq n-1} k_{i_1}\dots k_{i_r}
$$
(other authors normalize by a suitable constant, but this would not affect the discussion here). For $r=1$ we get indeed a multiple of the mean curvature, and for $r=n-1$ the Gauss-Kronecker curvature. In fact, $E_r$ is the  coefficient of  $x^{n-1-r}$ in the characteristic polynomial of $S$ (up to sign). Our second main result is the following. As usual, the writing
$\tilde H_f(t)\sim o(t^{\alpha})$ means that $\lim_{t\to 0} t^{-\alpha}\tilde H_f(t)=0$. For the proof, see Section \ref{phciso}. 

\begin{thm} \label{hciso}
\item (a) Let $\Omega$ be  a domain in $\real n$ or ${\bf H}^n$. If 
$\tilde H_f(t)\sim o(t)$ for all $ f\in C^{\infty}_0(\bd\Omega)$, then $\Omega$ is a geodesic ball.

\item (b) Now assume that $\Omega$ is a domain in $\sphere n$, and that, for some integer $k\geq 2$, 
$\tilde H_f(t)\sim o(t^{\frac k2})$ for all $ f\in C^{\infty}_0(\bd\Omega)$.
Then the  mean curvatures  $E_1,\dots,E_{k-1}$ are all constant on $\bd\Omega$. 

\item (c) In particular, if $\Omega\subseteq\sphere n$ and 
$\tilde H_f(t)\sim o(t^{\frac n2})$ for all $f\in C^{\infty}_0(\bd\Omega)$,
then $\bd\Omega$ is isoparametric (and consequently $\tilde H_f(t)$ is identically zero at all times). 
\end{thm}

Again, the different behavior in (a) (which is very easy to prove) as opposed to (b) and (c) (which are much more complicated) is due essentially to the validity of the Alexandrov theorem. 


\subsection{Relation with other overdetermined problems}\label{op}

Perhaps the seminal work in this field  was done by J. Serrin, who considered the following overdetermined problem, and inspired a good part of the following research:
\begin{equation}\label{serrin}
\twosystem
{\Delta v=1\quad\text{on}\quad\Omega,}
{v=0,\quad \derive v{\nu}=c\quad\text{on}\quad\bd\Omega,}
\end{equation}
where $c$ is a constant. Problem \eqref{serrin} is often referred  to as the {\it Serrin problem}. 
Domains supporting a solution to \eqref{serrin} are termed {\it harmonic} in \cite{R-S} because they are characterized by the following property: the mean value of a harmonic function on $\Omega$ equals its mean value on $\bd\Omega$. Serrin's celebrated result states that the only Euclidean domains which support a solution to \eqref{serrin} are balls. More generally, he proved in \cite{Ser} that the only Euclidean domains admitting a {\it positive} solution to the overdetermined problem
\begin{equation}\label{serring}
\twosystem
{\Delta v=F(v)\quad\text{on}\quad\Omega,}
{v=0,\quad \derive v{\nu}=c\quad\text{on}\quad\bd\Omega,}
\end{equation}
are balls. This rigidity result was later extended to the hyperbolic space and the hemisphere in \cite{K-P}. However,  on the whole sphere (and in any other manifold)  the classification problem is, to the best of our knowledge, still open (but see the next Section). 

\smallskip

Another famous problem  is the so-called {\it Schiffer problem (D)} for a given Dirichlet eigenvalue $\lambda$ (for more details see \cite{Shk}): 
\begin{equation}\label{schiffer}
\twosystem
{\Delta u=\lambda u\quad\text{on}\quad\Omega,}
{u=0, \, \derive{u}{\nu}=c\quad\text{on}\quad\bd\Omega.}
\end{equation}
We note that on a Euclidean ball any radial eigenfunction is a solution to \eqref{schiffer}; thus, there are solutions to the above problem for infinitely many eigenvalues $\lambda$. 
The following conjecture, proposed by Berenstein and often called {\it Schiffer conjecture (D)},  seems to be still open, even in Euclidean space:

\smallskip

{\bf Conjecture.} {\it Let $\Omega$ be a Euclidean domain and let $\lambda$ be any fixed Dirichlet eigenvalue of $\Omega$.  If $\Omega$ supports a solution to \eqref{schiffer} then $\Omega$ is a ball.}

\smallskip

There are several partial results related to this conjecture, see for example \cite{B-Y}. It is also known that the Neumann version of this conjecture (obtained by changing the boundary conditions in \eqref{schiffer} to $\derive{u}{\nu}=0, u=c$ and known as {\it Schiffer conjecture (N)}) is equivalent to the famous Pompeiu problem (the interested reader could consult \cite{B-S} and \cite{Wil}).

\smallskip

Now let us consider problem \eqref{schiffer} when $\lambda=\lambda_1(\Omega)$, the lowest Dirichlet eigenvalue of $\Omega$:

\begin{equation}\label{extremal}
\twosystem
{\Delta u=\lambda_1(\Omega) u\quad\text{on}\quad\Omega,}
{u=0, \, \derive{u}{\nu}=c\quad\text{on}\quad\bd\Omega.}
\end{equation}
In this case, domains for which a solution exists are called {\it extremal domains}: they are critical points of the first Dirichlet eigenvalue under volume preserving deformations of $\Omega$ (this follows from Hadamard's formula, see \cite{Sch} and \cite{ES-I}). 
It follows immediately from Serrin's result \cite{Ser}, and the existence of a positive eigenfunction associated to $\lambda_1(\Omega)$,   that the only Euclidean domains supporting a solution to \eqref{extremal} are balls: this by the way shows that the Schiffer  conjecture (D) above is true for the first eigenvalue. However, the classification of extremal domains in general Riemannian manifolds (in particular, in the sphere) is still an open problem. 
Interesting families of extremal domains of small prescribed volume are shown to exist near any nondegenerate critical point of the scalar curvature of  any Riemannian manifold: see \cite{P-S}.

\smallskip

We will observe in Section \ref{rop}  that the constant flow property implies existence of a solution to all of the above problems (see also \cite{Shk}).  That is:

\begin{thm}\label{implies}Let $\Omega$ be a Riemannian domain having the constant flow property. Then:

\item a) $\Omega$ supports a solution to the Serrin problem \eqref{serrin}; 

\item b)  $\Omega$ supports a solution to the Schiffer problem \eqref{schiffer} for infinitely many Dirichlet eigenvalues; 

\item c) $\Omega$ is an extremal domain (that is, it supports a solution to \eqref{extremal}).
\end{thm}
For the proof, see Section 2. 

Finally let us mention the results of  Magnanini and Sakaguchi in \cite{M-S}, related to  the function $u(t,x)$  defined in \eqref{temp}. A hypersurface $\Sigma$ contained in the interior of $\Omega$ is said to be {\it stationary isothermic} if it is isothermic for all times $t>0$; that is, if there exists a smooth function $\psi:(0,\infty)\to(0,\infty)$ such that:
$$
u(t,x)=\psi(t)\quad\text{for all $t>0$ and $x\in \Sigma$.}
$$
The authors then show that if $\Omega$ is a bounded, convex domain of $\real n$ admitting a stationary isothermic hypersurface, then $\Omega$ is a ball (the result continues to hold under less restrictive assumptions on the boundary of $\Omega$, see \cite{M-S} for more details). This problem could also be seen as an overdetermined problem for the heat equation.


\subsection{The isoparametric property}\label{ib}

We have seen that classification theorems for the Serrin problem \eqref{serrin} and for extremal domains have been proved so far only when the ambient manifold is Euclidean space, the hyperbolic space or the hemisphere: therefore, the natural question is whether, on the whole sphere, there exist other "exotic" examples (that is, examples not isometric to geodesic balls).  

\smallskip

The answer is actually affirmative, and a first family of such examples was constructed by Berenstein in \cite{Ber2}: this is the family  of domains in $\sphere n$ bounded by certain Clifford tori (tubes around a great circle  in $\sphere n$).   In fact, these domains admit solutions to the Schiffer problem \eqref{schiffer} for infinitely many eigenvalues.  

\smallskip

The matter was later expanded and clarified in \cite{Shk}, where it is observed that if a domain in $\sphere n$  has a connected, isoparametric boundary then it supports a solution to the Serrin problem and also to the Schiffer problem $(D)$ for infinitely many eigenvalues (see \cite{Shk}, Theorem 2 p. 549). Moreover, it has the constant flow property in the sense of Definition \ref{chf} (as proved in Section 5.3 of \cite{Shk}).  Thus, the sphere hosts a large variety of new examples. Let us then recall the general definition of the isoparametric property. 

\smallskip

Let $M$ be a  Riemannian manifold  and $U$ an open subset of $M$. A smooth function $F: U\to \reals$ is called {\it isoparametric} if there exist smooth functions 
$A$ and $B$ defined on the range of $F$ such that:
\begin{equation}\label{iso}
\twosystem
{\Delta F=A\circ F,}
{\abs{\nabla F}^2=B\circ F.}
\end{equation}
Then, the (smooth) hypersurface $\Sigma$ of $U\subseteq M$ is called {\it isoparametric} if it is a regular level set of an isoparametric function.  In fact, any isoparametric function defines a whole one-parameter family of isoparametric hypersurfaces, and any two members of the family are at constant distance to each other.  For the main facts on isoparametricity see the standard reference \cite{Tho} and also \cite{Wan}. Let us observe some equivalent, more geometric, definitions.  If $\Sigma$ is  a smooth hypersurface of $M$ and $\rho:M\to\reals$ is the distance function to $\Sigma$, then  the level sets (equidistants) $\rho^{-1}(r)$ are smooth provided that $r<\epsilon$ is small enough. We have the following characterizations; the first follows easily from the definition \eqref{iso}, while the second is due to Cartan \cite{Car}.

\begin{thm} \label{cartan} \item a) The hypersurface $\Sigma$ is isoparametric if and only if all equidistants  sufficiently close to $\Sigma$ have constant mean curvature.

\item b) $\Sigma$ is isoparametric in a space form $M$  if and only if it has constant principal curvatures (that is, the characteristic polynomial of the shape operator of $\Sigma$ is the same at all points).
\end{thm} 

As proved by Cartan,  the only compact isoparametric hypersurfaces of Euclidean or hyperbolic space  are geodesic balls; on the other hand, on the sphere the situation is much more interesting, and has generated  deep mathematical research, starting from Cartan himself. After the work of M\"unzner \cite{Mun} we know that isoparametric hypersurfaces of the sphere are given by level sets of restrictions to $\sphere n$ of certain (globally defined) homogeneous polynomials in $\real {n+1}$, called {\it Cartan-M\"unzner  polynomials},  and that the number $g$ of distinct principal curvatures can be only $g=1,2,3,4$ and $6$. Families of isoparametric hypersurfaces have been constructed in each of the above cases: for example, Clifford tori correspond to $g=2$ and, for $g=4$, there exist examples with non-homogeneous boundary. 

\smallskip

Let us briefly explain why a spherical domain $\Omega$ bounded by a connected isoparametric boundary $\Sigma$ has the constant flow property. From the general theory one knows that $\Omega$ is a smooth tube of constant radius around a smooth (minimal) submanifold $N$ of codimension at least two (the focal variety of $\Sigma$). The crucial fact is that the equidistants from $N$ (hence, also the equidistants from $\bd\Omega=\Sigma$) all have constant mean curvature. Now, if one defines radial functions as those functions which are constant on the equidistants, then one can verify that the Laplacian of $\Omega$ takes radial functions to radial functions. In turn, this implies that the solution of the heat equation with radial initial data (in particular, our function $u_t$) will stay radial at all times: as a consequence, its normal derivative $\bd u_t/\bd\nu$ (the heat flow)  will be constant on the boundary for all times. Then, any such domain has the constant flow property. 

\smallskip

The approach followed in \cite{Shk}, Section 5.3 to prove this fact is to work directly with definition \eqref{iso} and to use the Fourier series representation of  $u(t,x)$ (see equation (5.8)).

\smallskip

We conclude the section by asking whether the existence of a solution to the Serrin problem, or to the other overdetermined problems examined above, in spaces different from the Euclidean space, the hyperbolic space or the hemisphere (where the answer is known) would imply some kind of isoparametric property of the boundary. Also, is it true that any domain admitting a solution to the Serrin problem must also have the constant flow property ? This would be a converse to Theorem \ref{implies}. There is no immediate reason to have a positive answer; however we don't have, at the moment, any specific counterexample.


\subsection{A general theorem on Riemannian manifolds} \label{gen} Theorem \ref{main} will follow from a more general result: Theorem \ref{general} below.  Let then $\Omega$ be a compact domain with smooth boundary in a Riemannian manifold $M$ and let $\rho:\Omega\to\reals$ now denote the distance function to the boundary of $\Omega$:
$
\rho(x)={\rm dist}(x,\bd\Omega).
$
As the boundary is smooth, there exist a small $\epsilon>0$ such that the function $\rho$ will be smooth in the $\epsilon$-tubular neighborhood $U$ of $\bd\Omega$:
\begin{equation}\label{collar}
U=\{x\in\Omega: 0\leq\rho(x)< \epsilon\}
\end{equation}
(precisely, when $\epsilon$ is small enough so that  $U$ does not meet the cut-locus of the normal exponential map). It is also well-known (and easy to verify) that, at each point $x\in U$ at distance $\rho(x)=r$ to the boundary, the level set $\rho^{-1}(r)$ is smooth and one has:
$$
\Delta\rho(x)=\text{trace of the second fundamental form of the equidistant $\rho^{-1}(r)$.}
$$
In other words, $\Delta\rho$ measures the mean curvature of the interior parallels (equidistants). Let us write, for simplicity:
$$
\eta\doteq\Delta\rho
$$
thus obtaining a smooth function on $U$. Note that the vector field $\nu=\nabla\rho$ is smooth on $U$ and is everywhere normal to the level sets $\rho^{-1}(r)$; when restricted on $\bd\Omega$, it will give the inner unit normal field. 

\begin{thm}\label{general} Let $\Omega$ be domain with smooth boundary in a Riemannian manifold $M$. Assume that $\Omega$ has the constant flow property. Then
$
\deriven{k}{\eta}{\nu}
$
is constant on $\bd\Omega$
for all $k=0,1,2,\dots$.
\end{thm}

The proof will be done in  Section 3.  Note that the result holds in any smooth (not necessarily analytic) Riemannian manifold. We can now show how Theorem \ref{main} follows easily from Theorem \ref{general}.


\subsection{Proof of Theorem \ref{main}}\label{pmain}

By assumption, the ambient manifold  $M$ is analytic. 
By the regularity results in \cite{K-N} and  \cite{Wil2} we know that, if  a domain admits a solution to the Schiffer problem \eqref{schiffer} for some eigenvalue $\lambda$, then its boundary must be analytic. As proved in Theorem \ref{implies},  any domain having the constant flow property 
admits a solution to the Schiffer problem (for infinitely many eigenvalues). Then, $\bd\Omega$ is analytic.

Consider the tubular neighborhood $U=\{\rho<\epsilon\}$ of  $\bd\Omega$ as in \eqref{collar}.
We will show that the nearby equidistants $\rho^{-1}(r)$ have constant mean curvature for all $r <\epsilon$, or, equivalently, that $\eta$ is constant on $\rho^{-1}(r)$ for $r<\epsilon$. Consider the diffeomorphism:
$$
\Phi: [0,\epsilon)\times\bd\Omega \to U
$$
defined by $\Phi(r,y)=\exp_{y}(r\nu(y))$, where $\nu(y)$ is the inner unit normal at $y\in\bd\Omega$. The pair $(r,y)$ gives rise to the {\it normal coordinates} of a point of $U$. As both $M$ and $\bd\Omega$ are analytic, the normal exponential map, hence also the map $\Phi$, must be analytic. Now, the composition of $\Phi^{-1}$ with the projection onto $[0,\epsilon)$, which is precisely the distance function $\rho$ on $U$, is also analytic, hence its Laplacian (the function $\eta$), is analytic. Let $x\in U$ be a point at distance $r$ to the boundary, and let $y$ be the foot of the minimizing geodesic segment from $x$ to $\bd\Omega$, so that $(r,y)$ are the normal coordinates of $x$.  As $\eta$ is analytic, $\eta(x)=\eta(r,y)$ equals the sum 
of its $r$-Taylor series based at $(0,y)$. Hence:
$$
\eta(x)=\sum_{k=0}^{\infty}\dfrac{1}{k!}\deriven{k}{\eta}{\nu}(y)r^k.
$$
By Theorem \ref{general}, one has $\deriven{k}{\eta}{\nu}(y)=a_k$ for all $k$, where $a_k$ is independent on $y$; this shows that the right-hand side of the previous equation does not depend on $y$, but only on $r=\rho(x)$:  hence $\eta$ is constant on 
$\rho^{-1}(r)$ and $\bd\Omega$ is isoparametric. 


\subsection{Plan of the paper and scheme of the proof}\label{pop}

\nero In Section \ref{rop} we first show that the constant flow property  is equivalent to a vanishing condition for the heat content function with (zero mean) boundary data. We then verify  that a domain with the constant flow property supports a solution to the Serrin problem (Theorem \ref{solserrin}), and also a solution to the Schiffer problem for infinitely many eigenvalues (Theorem \ref{solschiffer}).

\nero In Section \ref{ahf} we start the proof of Theorem \ref{general}, by recalling the main results on the asymptotic expansion of the heat content and the heat flow proved in \cite{Savo1} and \cite{Savo2}. 
These results will be needed for the proof of Theorem \ref{general}. In fact, it turns out that the heat flow at a point $y\in\bd\Omega$ admits a complete asymptotic series, as $t\to 0$, of the following type:
\begin{equation}\label{asflow}
\derive{u}{\nu}(t,y)\sim\dfrac{1}{\sqrt\pi}\cdot\dfrac{1}{\sqrt t}+\sum_{k=0}^{\infty}(1+\frac k2) B_{k+2}(y)\cdot t^{k/2}
\end{equation}
for a sequence of smooth invariants $B_k(y)\in C^{\infty}(\bd\Omega)$. Clearly, if $\Omega$ has the constant flow property then all these invariants must be constant functions on $\bd\Omega$:  from this information, after some work, one can eventually derive that the normal derivative, of arbitrary order, of the function $\eta$ is constant on the boundary.

\smallskip

More in detail, it follows from the results in \cite{Savo2}  that the above invariants can be written as $B_k=-\bar D_k\eta$; here $\bar D_k$ is a certain differential operator belonging to the algebra generated by the Laplacian $\Delta$ of $\Omega$ and by the operator $N$ acting on $\phi\in C^{\infty}(U)$  as follows:
$$
N\phi=2\scal{\nabla\phi}{\nabla\rho}-\phi\Delta\rho=2\derive{\phi}{\nu}-\eta\phi
$$
where $\rho$ is the distance function to $\bd\Omega$ and where we have set $\nu=\nabla\rho$. The operator $\bar D_k$ can be computed using an explicit recursive scheme defining  some related operators $D_k$: this scheme was proved in \cite{Savo1} and will be recalled in Theorem \ref{savoone}.  The proof of Theorem \ref{general} will be  by induction on the order $k$ of the normal derivative of $\eta$, and in Section \ref{two}  we illustrate the general strategy by proving Theorem \ref{general} for $k\leq 2$.

\smallskip

\nero In Section \ref{sgeneral} we prove Theorem \ref{main} and Theorem \ref{hciso}. We first prove Proposition \ref {level}, in which we relate the invariants $B_k$ in \eqref{asflow} with the normal derivatives of $\eta$. More precisely we prove that,  if the function $\eta$ restricts to a constant function on $\bd\Omega$ together with all of its normal derivatives up and including the order $k$ 
then, for all $y\in\bd\Omega$ one has:
$$
B_{k+3}(y)=-w(\bar D_{k+3})\deriven{k+1}{\eta}{\nu}(y)+b_k
$$
where $b_k$ is a constant which does not depend on $y\in\bd\Omega$.  Here $w(\bar D_{k+3})$   (the so-called {\it weight} of the operator $\bar D_{k+3}$) is the coefficient of the  highest order normal derivative in $\bar D_{k+3}$. Assume that $w(\bar D_{k+3})$ is non-zero for all $k$. Then, an easy inductive argument (Proposition \ref{inductive}) shows that, if $B_2,\dots,B_{m}$ are constant on $\bd\Omega$ for some $m\geq 2$, then the functions $\eta,\derive{\eta}{\nu},\dots, \deriven{m-2}{\eta}{\nu}$ will also be  constant on $\bd\Omega$. The proof of Theorem \ref{general} now follows immediately, while the proof of Theorem \ref{hciso} requires an  additional argument involving the Newton identities (see Section \ref{phciso}). 

In conclusion, all the results of the paper will be completely proved once we show that $w(\bar D_{k+3})$ is non-zero for all $k$; thus in Theorem \ref{combinatorial} of Section \ref{sgeneral} we state the main combinatorial result, giving the explicit expression of the weight of $\bar D_k$ for all $k$. 

\nero In the remaining sections we prove Theorem \ref{combinatorial}, by applying the recursive scheme which define the operators $\bar D_k$. For convenience, we have divided the proof in several sections.
The proof is progressively reduced to a set of combinatorial identities for the so-called Hankel transforms associated to a certain numerical sequence (see Lemma \ref{mainlemma}); these identities can be derived from the work of Tamm in \cite{Tamm}, and  we give an explicit account of that  in the Appendix (Section \ref{appendix}).


\section{Equivalent condition and other overdetermined problems} \label{rop}

Given a smooth function $\phi(x)$ on $\Omega$, we let $\phi_t(x)\doteq\phi(t,x)$ be the solution of the heat equation with initial data $\phi_0(x)=\phi(x)$ and Dirichlet boundary conditions:

\begin{equation}\label{hefi}
\threesystem
{\Delta \phi_t+\derive {\phi_t}t=0}
{\phi_0=\phi\quad\text{on}\quad\Omega}
{\phi_t=0 \quad\text{on $\bd\Omega$, for all $t>0$}.}
\end{equation}
The total heat inside the domain at time $t$ is measured by the {\it heat content} function, defined for $t\geq 0$ by:
$$
H_{\phi}(t)=\int_{\Omega}\phi_t.
$$
$H_{\phi}(t)$ is smooth for $t>0$ but only continuous at $t=0$; it will also be called the {\it heat content with initial data $\phi$}. In what follows, we consider the following spaces:
$$
C^{\infty}_0(\bd\Omega)=\Big\{f\in C^{\infty}(\bd\Omega): \int_{\bd\Omega}f=0\Big\},
\quad {\cal H}_0(\Omega)=\Big\{\phi\in C^{\infty}(\Omega): \Delta\phi=0, \int_{\bd\Omega}\phi=0\Big\}.
$$

Recall that, given $f\in C^{\infty}(\bd\Omega)$ we denoted by $\tilde f_t$ the solution of the heat equation with zero initial conditions and boundary conditions prescribed by $f$: see \eqref{pbc}. It is clear that, if $\phi$ is the harmonic extension of $f$ to $\Omega$ (that is, if $\phi$ satisfies $\Delta\phi=0$ on $\Omega$ and $\phi=f$ on $\bd\Omega$) then $\tilde f_t$ can be written:
$$
\tilde f_t=\phi-\phi_t
$$
for all $t>0$. Integrating on $\Omega$, we see  that the heat content with boundary data $f$, that is, the function  
$\tilde H_f(t)=\int_{\Omega}\tilde f_t$, can be written:
\begin{equation}\label{relationhc}
\tilde H_f(t)=\int_{\Omega}\phi-H_{\phi}(t),
\end{equation}
where $\phi$ is the harmonic extension of $f$ to $\Omega$.

\smallskip

We will often use the fact that the only functions   on $\bd\Omega$ which are $L^2-$orthogonal to $C^{\infty}_0(\bd\Omega)$ are the constants.

\begin{thm}\label{zeroflow}
A domain $\Omega$ has the constant flow property  if and only if:

\item (a)  the heat content with initial data in ${\cal H}_0(\Omega)$ is identically zero at all times; that is, one has
$
H_{\phi}(t)=0
$
for all $\phi\in {\cal H}_0(\Omega)$ and for all $t\geq 0$.

\item (b) the heat content with boundary data in $C^{\infty}_0(\bd\Omega)$ is identically zero at all times; that is, one has
$
\tilde H_f(t)=0
$
for all $f\in C^{\infty}_0(\bd\Omega)$ and for all $t\geq 0$.
\end{thm}

\begin{proof} (a) We first observe that we have the identity, valid for all $t>0$:
\begin{equation}\label{phiut}
H_{\phi}(t)=\int_{\Omega}\phi u_t,
\end{equation}
where $u_t$ is, as usual,  the solution of our original equation \eqref{temp}.
In fact, if $k(t,x,y)$ denotes the heat kernel of $\Omega$ with Dirichlet boundary conditions, one has $\phi_t(x)=\int_{\Omega}k(t,x,y)\phi(y)dy$ hence:
$$
\int_{\Omega}\phi_t(x)dx=\int_{\Omega}\Big(\int_{\Omega}k(t,x,y)\phi(y)dy\Big)\,dx=
\int_{\Omega}\phi(y)\Big(\int_{\Omega}k(t,x,y)dx\Big)\,dy=\int_{\Omega}\phi(y)u_t(y)\,dy.
$$
If $\phi$ is harmonic on $\Omega$ we obtain, from \eqref{phiut}, the Green formula and the fact that $u_t$ vanishes on the boundary:
\begin{equation}\label{flow}
H'_{\phi}(t)=-\int_{\Omega}\phi\Delta  u_t=-\int_{\bd\Omega}\phi\derive{u_t}{\nu}.
\end{equation}
Now assume that $\phi\in {\cal H}_0(\Omega)$: then $\phi$ is harmonic  and $\int_{\bd\Omega}\phi=0$. If $\Omega$ has the constant flow property then $\derive{u_t}{\nu}$ is constant on $\bd\Omega$, hence it can be taken out of the integral  \eqref{flow} so that $H'_{\phi}(t)=0$ for all $t>0$. Therefore, for all $t$:
$$
H_{\phi}(t)=H_{\phi}(0)=\int_{\Omega}\phi.
$$
Now, since $\phi_t$ converges uniformly to zero as $t\to\infty$,   so does $H_{\phi}(t)$, which in turn implies that $\int_{\Omega}\phi=0$: hence $H_{\phi}(t)=0$ for all $t$.

\smallskip

Conversely, assume that $H_{\phi}(t)=0$  for all $t\geq 0$ and $\phi\in {\cal H}_0(\Omega)$. It is enough to show that
$$
\int_{\bd\Omega}f\derive{u_t}{\nu}=0
$$
for all $f\in C^{\infty}_0(\bd\Omega)$ and $t>0$, because then $\derive{u_t}{\nu}$ must be constant on $\bd\Omega$. Fix $f\in C^{\infty}_0(\bd\Omega)$ and consider the unique harmonic function $\phi$ which extends $f$ to $\Omega$. By assumption $\phi\in {\cal H}_0(\Omega)$ hence $H_{\phi}(t)=0$ for all $t$. 
But then $H'_{\phi}(t)=0$ for all $t$ and by \eqref{flow} we see
$$
0=-\int_{\bd\Omega}\phi\derive{u_t}{\nu}=-\int_{\bd\Omega}f\derive{u_t}{\nu},
$$
hence the assertion.

\smallskip

{\it Proof of (b)}.  It follows immediately from (a) and \eqref{relationhc}. It can also be proved by observing that, if $f\in C^{\infty}_0(\bd\Omega)$ and if $\phi$ is its harmonic extension to $\Omega$, then, by \eqref{relationhc} and \eqref{flow}:
\begin{equation}\label{tildehprime}
\tilde H'_{f}(t)=\int_{\bd\Omega} f\derive{u_t}{\nu}.
\end{equation}
One can then argue as before.

\end{proof} 

 We can now show the following fact. 

\begin{thm} \label{solserrin} Any Riemannian domain $\Omega$ having the constant flow property supports a solution to the Serrin problem \eqref{serrin}.
\end{thm}

\begin{proof} Assume that $\Omega$ has the constant flow property  and let $v$ be the unique function such that  $\Delta v=1$  on $\Omega$ and $v=0$ on the boundary. We have to show that its normal derivative is constant on $\bd\Omega$. 
Fix $f\in C^{\infty}_0(\bd\Omega)$ and extend $f$ to a harmonic function $\phi$ on $\Omega$. Then, by Theorem \ref{zeroflow} we know that the  heat content with initial data $\phi$ is identically zero: $\int_{\Omega}\phi_t=0$ for all $t\geq 0$.  Taking $t=0$ we see that $\phi$ integrates to zero on $\Omega$. Therefore, by the Green formula:
$$
0=\int_{\Omega}\phi=\int_{\Omega}\phi\Delta v=\int_{\bd\Omega}\phi\derive {v}{\nu}=\int_{\bd\Omega}f\derive {v}{\nu}.
$$ 
As $f\in C^{\infty}_0(\bd\Omega)$ was arbitrary, one must have $\derive{v}{\nu}={\rm const}$. 
\end{proof}


\begin{thm}\label{solschiffer} Let $\Omega$ be a domain with the constant flow property. Then the overdetermined problem
$$
\twosystem
{\Delta u=\lambda u\quad\text{on}\quad\Omega}
{u=0, \, \derive{u}{\nu}=c\quad\text{on}\quad\bd\Omega}
$$
admits a solution for an infinite sequence $\{\lambda_j^+\}$ of eigenvalues, in particular, for $\lambda=\lambda_1(\Omega)$. Hence, any domain with the constant flow property is also extremal. 
\end{thm} 

\begin{proof} Let ${\rm Spec}(\Omega)=\{\lambda_1,\lambda_2,\dots\}$ be the spectrum of $\Omega$ for the Dirichlet boundary conditions. We single out an (infinite) subset  ${\rm Spec}_+(\Omega)\subseteq {\rm Spec}(\Omega)$ , as follows:
$$
{\rm Spec}_+(\Omega)=\Big\{\lambda\in {\rm Spec}(\Omega): \text{there exists an eigenfunction $\phi\in V(\lambda)$ such that $\int_{\Omega}\phi\ne 0$}\Big\}.
$$
As any eigenfunction associated to $\lambda_1$ does not change sign, we have  that $\lambda_1\in {\rm Spec}_+(\Omega)$. The aim is to prove that
if a domain $\Omega$ has the constant flow property then  the problem at hand (Schiffer problem (D))  admits a solution for all $\lambda\in 
{\rm Spec}_+(\Omega)$. Let us list the elements of  ${\rm Spec}_+(\Omega)$ (in increasing order) as follows:
$$
{\rm Spec}_+(\Omega)=\{\lambda_1^+,\lambda_2^+,\dots\}
$$
Note that $\lambda_1^+=\lambda_1$. 
Let us briefly justify why this subset is actually infinite. Writing  the Fourier series expansion of the constant function $1$, 
one sees that only the eigenvalues in ${\rm Spec}_+(\Omega)$ contribute with a non-zero term: if this set were finite then the Fourier series would  also be finite, which is impossible because otherwise one would have $1=0$ on the boundary.

We now construct a special orthonormal basis of eigenfunctions. Given $\lambda_j^+\in{\rm Spec}_+(\Omega)$ consider the linear map $I:V(\lambda_j^+)\to\reals$ given by integration over $\Omega$:
$$
I\phi=\int_{\Omega}\phi.
$$
As the orthogonal complement of the kernel of $I$ is one-dimensional, we can select an orthonormal basis of $V(\lambda_j)$ as follows:
$$
\{\phi_j^+,\psi_1,\dots,\psi_m\}
$$
where $\phi_j^+$ has a positive integral over $\Omega$, while all the other eigenfunctions $\psi_j$ (if $m\geq 1$) have zero integral. We do this for every element in ${\rm Spec}_+(\Omega)$; for the eigenvalues which do not belong to ${\rm Spec}_+(\Omega)$, we take any orthonormal basis of the respective eigenspace. Repeating the procedure for all eigenvalues, we obtain a special orthonormal basis of $L^2(\Omega)$.

Let ${\rm Spec}(\Omega)=\{\lambda_1,\lambda_2,\dots\}$ be the full Dirichlet spectrum of $\Omega$ (eigenvalues are repeated according to multiplicity), and select any othonormal basis of associated eigenfunctions $\{\phi_j\}_{j=1,2,\dots}$. Then, the heat kernel for the Dirichlet conditions is:
$$
k(t,x,y)=\sum_{j=1}^{\infty}e^{-\lambda_jt}\phi_j(x)\phi_j(y),
$$
and the heat content with  initial data $f$ has the following Fourier expansion:
$$
H_f(t)=\sum_{j=1}^{\infty}e^{-\lambda_jt}\int_{\Omega}\phi_j\cdot\int_{\Omega}f\phi_j.
$$
We now adopt the special orthonormal basis constructed above: it is clear that only the eigenvalues in ${\rm Spec_+}(\Omega)$ show up
\begin{equation}\label{fe}
H_f(t)=\sum_{j=1}^{\infty} e^{-\lambda_j^+t}\int_{\Omega}\phi_j^+\cdot\int_{\Omega}f\phi_j^+
\end{equation}
and that there is only one  term containing each $\lambda_j^+\in {\rm Spec}_+(\Omega)$. 

We can now prove the Theorem. First assume that $\Omega$ has the constant flow property and fix $\psi\in {\cal H}_0(\Omega)$. From the previous theorem, we know that  $H_{\psi}(t)=0$ for all $t$. From \eqref{fe} we easily get (since there is only one term corresponding to $\lambda_j^+$):
$$
\int_{\Omega}\phi_j^+\cdot\int_{\Omega}\psi\phi_j^+=0
$$
for all $j$ and  in turn
$$
\int_{\Omega}\psi\phi_j^+=0
$$
for all $j$ because $\int_{\Omega}\phi_j^+$ is positive by definition. Fix $j$. As $\psi$ is harmonic on $\Omega$  we see, by the Green formula:
$$
\int_{\bd\Omega}\psi\derive{\phi_j^+}{\nu}=\lambda_j^+\int_{\Omega}\psi\phi_j^+=0.
$$
The above holds for all $\psi\in{\cal H}_0(\Omega)$; as any element of $C^{\infty}_0(\bd\Omega)$ 
is the restriction to the boundary of an element of ${\cal H}_0(\Omega)$, the above immediately implies that $
\derive{\phi_j^+}{\nu}
$
is constant on $\bd\Omega$ which proves that the pair $(\lambda_j^+, \phi_j^+)$ is a solution of the Schiffer problem (D). 

\medskip

More precisely, we have proved that a domain $\Omega$ has the constant flow property if and only if the pair $(\lambda_j^+,\phi_j^+)$ is a solution of the Schiffer problem (D) for all $\lambda_j^+\in {\rm Spec}_+(\Omega)$.
\end{proof}



\section{Asymptotics of the heat flow: review}\label{ahf}

In this section we will review the main results on the asymptotics of the heat flow, which were proven in \cite{Savo1} and \cite{Savo2}, and which will be used in this paper. 

\smallskip

Given a smooth function $\phi\in C^{\infty}(\Omega)$, we let $\phi_t(x)$ be the solution of the heat equation with initial data $\phi$ and Dirichlet boundary conditions, as in \eqref{hefi}, and we consider the associated heat content function $H_{\phi}(t)=\int_{\Omega}\phi_t$.
 It was first observed by van den Berg and Gilkey in \cite{vdB-G}  that the heat content admits an asymptotic series, as $t\to 0$, of type:
\begin{equation}\label{as}
H_{\phi}(t)\sim \int_{\Omega}\phi-\sum_{k=1}^{\infty}\beta_{k}(\phi)t^{k/2}.
\end{equation}
for a family of invariants $\beta_{k}(\phi)\in \reals$. The authors then computed the coefficients $\beta_k(\phi)$ up to $k=4$
(see also related work for the inhomogeneous case in \cite{vdB-G2}). In \cite{Savo1} a recursive formula for the calculation of the whole asymptotic series \eqref{as} was given: let us explain the outcome. As in Section \ref{gen}, we fix a tubular neighborhood $U$ of $\bd\Omega$ where the distance function $\rho$ to the boundary of $\Omega$ is smooth and write $\nu=\nabla\rho$, a smooth vector field on $U$ which restricts to the unit normal vector on $\bd\Omega$.  Consider the operator $N$ acting on $f\in C^{\infty}(U)$ as follows:
$$
\begin{aligned}
Nf&=2\scal{\nabla f}{\nabla\rho}-f\Delta\rho\\
&=2\derive{f}{\nu}-\eta f
\end{aligned}
$$
Now let 
\begin{equation}\label{algebra}
{\cal A}={\cal A}(N,\Delta)
\end{equation}
be the algebra of differential operators acting on $C^{\infty}(U)$ and generated by the operator $N$ (of degree one), and the Laplacian $\Delta$ (of degree two).
Then, ${\cal A}$ comes with a natural grading given by the degree, and each element of $\cal A$ will be a (non-commutative) polynomial in $N$ and $\Delta$. 
The main result of \cite{Savo1} states that there is a sequence $\{D_k\}$ of differential operators in the algebra $\cal A$ such that the coefficient $\beta_k(\phi)$ is obtained by integrating the function $D_k\phi$ over the boundary. The sequence $\{D_k\}$ is recursively defined, as follows. 
\smallskip

We start by defining  operators $R_{kj},S_{kj}\in {\cal A}(N,\Delta)$, depending on two non-negative integers $k,j$, by the following recursive rule:
\begin{equation}\label{rsone}
\threesystem
{R_{kj}=-(N^2+\Delta)R_{k-1,j}+NS_{k-1,j}}
{S_{kj}=\Delta NR_{k-1,j}-\Delta S_{k-1,j}+NR_{k-1,j-1}}
{R_{00}=I, S_{00}=0, R_{kj}=S_{kj}=0 \quad\text{if $k<0$ or $j<0$}.}
\end{equation}
Now set:
$\{a,b\}\doteq\dfrac{\Gamma (a+b+\frac12)}{(a+b)!\Gamma (a+\frac 12)}$, and define the operators $Z_{n+1}, \alpha_n\in{\cal A}(N,\Delta)$ by:
\begin{equation}\label{za}
Z_{n+1}=\sum_{j=0}^n\{n+1,j-1\}R_{n+j,j}, \quad \alpha_n=\sum_{j=0}^{n+1}\{n,j\}S_{n+j,j}.
\end{equation}

This is Theorem 2.1 in \cite{Savo1}.

\begin{thm}\label{savoone}
Let $\beta_k(\phi)$ be the coefficient of  $t^{k/2}$ in the asymptotic expansion of the heat content \eqref{as}, and let $D_k\in{\cal A}(N,\Delta)$ be the homogeneous polynomial of degree $k-1$ defined inductively by the formulas:
\begin{equation}\label{recurrence}
\threesystem
{D_1=\dfrac{2}{\sqrt\pi}I}
{D_{2n}=\dfrac{1}{\sqrt\pi}\sum_{i=1}^n\dfrac{\Gamma(i+\frac 12)\Gamma(n-i+\frac 12)}{n!}D_{2i-1}\alpha_{n-i}}
{D_{2n+1}=\dfrac{1}{\sqrt\pi}Z_{n+1}+\dfrac{1}{\sqrt\pi}\sum_{i=1}^n\frac{i!\Gamma(n-i+\frac 12)}{\Gamma(n+\frac32)}D_{2i}\alpha_{n-i}}
\end{equation}
Then, for all $k\geq 1$, we have 
$
\beta_{k}(\phi)=\int_{\bd\Omega}D_k\phi.
$
\end{thm}

The  sequence $\{D_k\}$ will be called the sequence of {\it heat content operators} of $\Omega$. We give below the explicit expression of the operators $D_1,\dots,D_6$ (taken from Table 2.2 in \cite{Savo1}). 
\begin{equation}\label{table}
\threesystem
{D_1=\dfrac{2}{\sqrt\pi}Id, \quad D_2=\dfrac12 N, \quad D_3=\frac{1}{6\sqrt\pi}(N^2-4\Delta), \quad
D_4=-\dfrac1{16}(\Delta N+3N\Delta)}
{D_5=-\dfrac1{240\sqrt\pi}\Big(N^4+16N^2\Delta+8N\Delta N-48\Delta^2\Big)}
{D_6=\dfrac{1}{768}\Big(\Delta N^3-N^3\Delta+N\Delta N^2-N^2\Delta N+40N\Delta^2+8\Delta^2N+16\Delta N\Delta\Big)}
\end{equation}
These results were refined in the paper \cite{Savo2} to obtain an asymptotic expansion of the heat flow valid at each point of the boundary (see Theorem 2.1 in \cite{Savo2}).

\begin{thm}\label{phiflow} Let $\phi_t$ be the solution of the heat equation with initial data $\phi$ and Dirichlet boundary conditions, as in \eqref{hefi}. Then, for all $y\in\bd\Omega$, there is an asymptotic series:
$$
\derive{\phi}{\nu}(t,y)\sim\dfrac{\phi(y)}{\sqrt\pi}\cdot\dfrac{1}{\sqrt t}+\sum_{k=0}^{\infty}\tilde D_k\phi(y)\cdot t^{k/2}\quad\text{as}\quad t\to 0,
$$
where $\tilde D_k\in{\cal A}(N,\Delta)$ is the operator 
$
\tilde D_k=\Big(1+\frac{k}2\Big)D_{k+2}.
$
\end{thm}
Taking unit initial data $\phi=1$ we see that the corresponding solution will be $u_t$, as in \eqref{temp}. Then, at each point $y\in\bd\Omega$,  Theorem \ref{phiflow} will give the asymptotic expansion of the heat flow:
\begin{equation}\label{uflow}
\derive{u}{\nu}(t,y)\sim\dfrac{1}{\sqrt\pi}\cdot\dfrac{1}{\sqrt t}+\sum_{k=0}^{\infty}\tilde D_k 1(y)\cdot t^{k/2}.
\end{equation}
If the heat flow is constant on $\bd\Omega$ for all times $t>0$, then necessarily the function $\tilde D_k1$ (hence also the function $D_k1$), when restricted to the boundary, will be constant  for all $k$.   We summarize these considerations in the following 

\begin{thm}\label{bk} Assume that $\Omega$ has the constant flow property, and let $\{D_k\}$ be the sequence of heat content operators, as defined in \eqref{recurrence}.  Then the function
$$
B_k\doteq D_k1|_{\bd\Omega}
$$
is constant on $\bd\Omega$ for all $k$. 
\end{thm}


\subsection{Proof of Theorem \ref{general} for $k\leq 2$}\label{two}
We finish this section by showing that the constant heat flow assumption and the expression of $D_1,D_2,D_3,D_4$ as given in \eqref{table} will imply that $\eta,\derive{\eta}{\nu}$ and $\deriven{2}{\eta}{\nu}$ are constant on $\bd\Omega$. This will give a hint for the general proof of Theorem \ref{general} (which states that the normal derivative, of arbitrary order, of the function $\eta$  is constant on $\bd\Omega$). 

\smallskip 

Notice that as $N1=-\eta$ and $\Delta1=0$ we have, from table \eqref{table}:
\begin{equation}\label{bifour}
B_2=-\frac12\eta, \quad B_3=-\dfrac{1}{6\sqrt\pi}N\eta, \quad B_4=\dfrac{1}{16}\Delta\eta.
\end{equation}
We now observe a useful splitting of the Laplace operator in the neighborhood $U$. Given $f\in C^{\infty}(U)$, we have
\begin{equation}\label{splitting}
\Delta f=-\deriven{2}{f}{\nu}+\eta\derive{f}{\nu}+\Delta_T f
\end{equation}
where $\Delta_T f$, the {\it tangential Laplacian} of $f$, is defined as follows. For $x\in U$, let $\rho^{-1}(r)$ be the level set of $\rho$ containing $x$ (so that $\rho(x)=r$). Then $\Delta_T f(x)$
is the Laplacian (for the induced metric on $\rho^{-1}(r)$) of the restriction of $f$ to $\rho^{-1}(r)$.
We will call
\begin{equation}\label{rad}
\Delta_R f\doteq -\deriven{2}{f}{\nu}+\eta\derive{f}{\nu}
\end{equation}
the {\it radial} Laplacian of $f$, so that, on $U$: $\Delta f=\Delta_R f+\Delta_T f$. The splitting \eqref{splitting} is easily proved by working with orthonormal frames of type $(e_1,\dots,e_{n-1},\nu)$, so that $(e_1,\dots,e_{n-1})$ will be an orthonormal frame of the equidistant through the point.
Then, from formulae  \eqref{bifour} and the splitting \eqref{splitting} we obtain: 
\begin{equation}\label{part}
\threesystem
{B_2=-\frac 12\eta}
{B_3=-\frac{1}{6\sqrt\pi}\Big(2\derive{\eta}{\nu}-\eta^2\Big)}
{B_4=\frac 1{16}\Big(-\deriven {2}{\eta}{\nu}+\eta\derive{\eta}{\nu}+\Delta_T\eta\Big)}
\end{equation}
If $\Omega$ has the constant flow property, then Theorem \ref{bk} asserts that the functions $B_k$ are all constant on $\bd\Omega$: an obvious inductive argument will then show that the normal derivatives $\eta,\derive{\eta}{\nu}$ and $\deriven{2}{\eta}{\nu}$ are also constant on $\bd\Omega$, which is Theorem \ref{general} for $k\leq 2$.

\smallskip

Finally, we observe the following immediate proof of Theorem \ref{main} when $\Omega$ is a domain in Euclidean or hyperbolic space. In fact, from table \eqref{part}, we see that $\bd\Omega$ has constant mean curvature, hence, from a well-known result by Alexandrov, $\Omega$ must be a ball.


\section{Proof of Theorem \ref{general}} \label{sgeneral} In this section we write the invariants $B_k=D_k1|_{\bd\Omega}$ of Theorem \ref{bk}
in terms of the normal derivatives of the function $\eta$ (as we have done it in Section \ref{two} for $k\leq 4$) : this will be used to give an inductive proof of Theorem \ref{general}.

We start by writing the invariants $B_k$ in a more suitable way. Given an operator $A$ of degree at least one in the algebra ${\cal A}(N,\Delta)$ defined in \eqref{algebra}, we can decompose it as follows:
\begin{equation}\label{dec}
A=\bar AN+\tilde A\Delta,
\end{equation}
for uniquely defined operators $\bar A$ (with $\deg \bar A=\deg{A}-1$) and $\tilde A$ (with  $\deg\tilde A=\deg{A}-2$).
Clearly the map $A\to\bar A$ is linear, and for $A,B\in{\cal A}(N,\Delta)$ one has:
\begin{equation}\label{overline}
\overline{AB}=A\overline{B}.
\end{equation}
For example, from table \eqref{table} we see:
\begin{equation}\label{list}
\quad \bar D_2=\dfrac12 I, \quad \bar D_3=\frac{1}{6\sqrt\pi}N, \quad
\bar D_4=-\dfrac1{16}\Delta, \quad \bar D_5=-\frac{1}{240\sqrt\pi}(N^3+8N\Delta).
\end{equation}
As $N1=-\eta$ and $\Delta 1=0$ we see that  $D_k1=-\bar D_k\eta$. Then, Theorem \ref{bk}  becomes the following statement. 
\begin{prop} \label{bardkey} Assume that $\Omega$ has the constant flow property. Then the function
$$
B_k=-\bar D_k\eta|_{\bd\Omega}
$$
is constant on $\bd\Omega$ for all $k$.
\end{prop}

Our next task is to determine the coefficient of the normal derivative of highest order in a given homogeneous operator belonging to the algebra ${\cal A}(N,\Delta)$. Define a function, called the {\it weight}: 
$$
w:{\cal A}(N,\Delta)\to\reals
$$
by setting $w(I)=1, w(N)=2, w(\Delta)=-1$ and then extending $w$ to ${\cal A}$ as an algebra homomorphism.  For example, from list \eqref{list}:
$$
w(\bar D_2)=\frac 12, \quad w(\bar D_3)=\dfrac{1}{3\sqrt\pi}, \quad w(\bar D_4)=\dfrac1{16}, \quad w(\bar D_5)=\dfrac{1}{30\sqrt\pi}.
$$


From the decomposition $\Delta=\Delta_R+\Delta_T$ of \eqref{splitting} one sees why $\Delta$ should have  weight $-1$.
 The inductive step is based on the following fact.

\begin{prop}\label{level} Fix an integer $k\geq 0$ and assume that the function $\eta$, together with all of its normal derivatives up and including the order $k$, restricts to a constant function on $\bd\Omega$. Then:

\item (a) If $A$ is an operator in ${\cal A}(N,\Delta)$ homogeneous of degree $k+1$, then:
$$
A\eta|_{\bd\Omega}=w(A)\deriven{k+1}{\eta}{\nu}|_{\bd\Omega}+c_k,
$$
where $w(A)$ is the weight of $A$ and $c_k$ is constant on $\bd\Omega$. 

\item (b) If  $B_{k+3}\in C^{\infty}(\bd\Omega)$ is  the invariant of Proposition \ref{bardkey}, then, for all $y\in\bd\Omega$:
$$
B_{k+3}(y)=-w(\bar D_{k+3})\deriven{k+1}{\eta}{\nu}(y)+b_k
$$
where $b_k$ is a constant which does not depend on $y\in\bd\Omega$.
\end{prop}
We will give the proof of this proposition in Section \ref{plevel} below. 
The following result is the main combinatorial fact needed in the proof of Theorem \ref{general}; its proof will take the rest of the paper, starting from the  next section. 

\begin{thm}\label{combinatorial} Let $\{D_k\}$ be the sequence of heat content operators, and  consider the operators $\bar D_k$ defined in \eqref{dec}. Then, one has for all $n\geq 1$:
$$
\twosystem
{w(\bar D_{2n})=\frac{2}{4^n n!}}
{w(\bar D_{2n+1})=\dfrac{1}{\sqrt\pi}\cdot\dfrac{1}{2^{n-1}(2n+1)!!}}
$$
In particular, $w(\bar D_k)\ne 0$ for all $k$.
\end{thm}

We observe the following easy consequence of Proposition \ref{level} and Theorem \ref{combinatorial}.

\begin{prop}\label{inductive} Assume that $\Omega$ is a domain such that the invariants $B_2,\dots,B_k$ are constant on $\bd\Omega$ for some $k\geq 2$. Then the functions $\eta,\derive{\eta}{\nu},\dots, \deriven{k-2}{\eta}{\nu}$ are  constant on $\bd\Omega$.

\end{prop} 

\begin{proof} The proof is by induction on $k$. The statement is true for $k=2$, because $B_2=-\frac 12\eta$ (see table \eqref{bifour}). Assume that it is true for the integer $k$, and assume that $B_2,\dots,B_{k+1}$ are constant. We have to show that 
$\deriven{k-1}{\eta}{\nu}$ is constant. Now, we know that 
$\eta,\derive{\eta}{\nu},\dots, \deriven{k-2}{\eta}{\nu}$ are all constant by the inductive hypothesis; by  Proposition \ref{level}b  we have:
$$
B_{k+1}=a_k\deriven{k-1}{\eta}{\nu}+c_k,
$$
where $a_k$ is the weight of $-\bar D_{k+1}$, which is non-zero by Theorem \ref{combinatorial}, and $c_k$ is constant. This immediately shows that $\deriven{k-1}{\eta}{\nu}$ must also be constant and the assertion follows.

\end{proof}

We can now prove Theorem \ref{general} and Theorem \ref{hciso}. 


\subsection{Proof of Theorem \ref{general}}\label{pgeneral}

Assume that $\Omega$ has the constant flow property. Recall that we need to show that the $k$-th normal derivative of the function $\eta$ is constant on $\bd\Omega$ for all $k\geq 0$. By Proposition \ref{bardkey} the function $B_k$ is constant on $\bd\Omega$ for all $k$. Then, the conclusion follows immediately from Proposition \ref{inductive}.


\subsection{Proof of Theorem \ref{hciso}}\label{phciso}

We consider the heat content $\tilde H_f(t)$ with boundary data $f\in C^{\infty}_0(\bd\Omega)$, as defined in \eqref{hcb}. From \eqref{tildehprime} we know that
$$
\tilde H'_f(t)=\int_{\bd\Omega}f\derive{u_t}{\nu}.
$$
Substituting for $\bd u_t/\bd\nu$ the expansion  in \eqref{uflow},  integrating the asymptotic series term to term, 
and recalling that $\tilde D_k1=(1+\frac k2)D_{k+2}1=(1+\frac k2)B_{k+2}$ 
we see that, as $t\to 0$:
\begin{equation}\label{tildeh}
\tilde H_f(t)\sim \sum_{k=2}^{\infty}\int_{\bd\Omega}fB_k\cdot t^{\frac k2}
\end{equation}
 (the series starts with $k=2$ because $f$ integrates to zero on $\bd\Omega$).

\smallskip

{\it Proof of (a).} By assumption, $\tilde H_f(t)=o(t)$ as $t\to 0$ hence, from \eqref{tildeh}: $\int_{\bd\Omega}fB_2=0$. This happens for all $f\in C^{\infty}_0(\bd\Omega)$ hence $B_2$ must be constant. From table \eqref{bifour} we see that $\eta=-2B_2$ must be constant, hence, as $\bd\Omega$ is compact and has constant mean curvature, $\Omega$ must be a ball by the Alexandrov theorem, which is valid in $\real n$ and ${\bf H}^n$.

\smallskip

{\it Proof of (b).} We have by assumption that $\Omega\subseteq \sphere n$ and that $\tilde H_f(t)\sim o(t^{k/2})$ for some $k\geq 2$, and for all $f\in C^{\infty}_0(\bd\Omega)$. We must then prove that all mean curvatures of order $\leq k-1$ are constant on $\bd\Omega$. 

\medskip

{\bf Step 1.} {\it The invariants $B_2,\dots,B_k$ are constant on $\bd\Omega$.}

\medskip

This follows immediately because, from \eqref{tildeh} and our assumption, $\int_{\bd\Omega}fB_j=0$ for all $j\leq k$. 
\medskip

{\bf Step 2.} {\it The functions $\eta,\derive{\eta}{\nu},\dots, \deriven{k-2}{\eta}{\nu}$ are  constant on $\bd\Omega$.}

\medskip

This follows from Step 1 and Proposition \ref{inductive}.

\smallskip

Define a field of endomorphisms $S$ of the tangent space of the tubular neighborhood $U$ by:
$$
S(X)=-\nabla_X\nu,
$$
where $\nu=\nabla\rho$ is everywhere normal to the equidistants. Note that $S(\nu)=-\nabla_{\nu}\nu$ is identically zero because $\nu$ is of unit length and tangent to the normal geodesics. 
When restricted to the tangent space of $\rho^{-1}(r)$ (in particular, the tangent space of $\bd\Omega$)  the endomorphism $S$ is just the shape operator, hence, on $U$, one has $\eta=\tr S$. 
A straightforward calculation shows that
\begin{equation}\label{normal}
\nabla_{\nu}S=S^2+R_{\nu},
\end{equation}
where $R_{\nu}(X)=R(\nu,X)\nu$ and $R$ is the Riemann tensor of the ambient manifold $\Omega$. 

\medskip

{\bf Step 3.} {\it The traces $\tr(S), \tr(S^2),\dots, \tr(S^{k-1})$ are constant on $\bd\Omega$.}

\medskip

Since $B_2=-\frac 12\eta$ is constant on $\bd\Omega$, also $\tr(S)=\eta$ is constant on $\bd\Omega$. Now $\nabla_{\nu}$ commutes with taking traces, and one has  $\tr(R_{\nu})={\rm Ric}(\nu,\nu)$, where ${\rm Ric}$ denotes the Ricci tensor. If $\Omega\subseteq \sphere{n}$, taking traces in \eqref{normal}  one then gets:
$$
\derive{\eta}{\nu}=\tr(S^2)+n-1.
$$
Using  the Leibniz rule we have, from \eqref{normal}: $\derive{\tr(S^j)}{\nu}=j\tr(S^{j+1})+j\tr(S^{j-1})$ and we easily arrive at:
$$
\deriven{m}{\eta}{\nu}=m!\,\tr(S^{m+1})+\sum_{j=0}^{m-1}a_j\tr(S^j)
$$
where each $a_j$ is a  constant. Induction on $k$ shows that Step 2 implies Step 3. 

\medskip

{\bf Step 4.} {\it The mean curvatures $E_1,\dots,E_{k-1}$ are constant on $\bd\Omega$.}

\medskip

This fact is a consequence of Step 3 and the well-known Newton's identities, which relate the $r$-th elementary symmetric function of the eigenvalues $k_1,\dots,k_{n-1}$ with the sums of the powers of the eigenvalues, that is, the trace of the powers of $S$:
$
\tr(S^m)=\sum_{j=1}^{n-1}k_j^m.
$
Precisely, one has, for all $k$:
$$
kE_k=\sum_{i=1}^k(-1)^{i-1}E_{k-i}\tr(S^i).
$$
Hence we have a scheme like this:
$$
\begin{aligned}
E_1&=\tr(S)\\
E_2&=E_1\tr(S)-\tr(S^2)\\
E_3&=E_2\tr(S)-E_1\tr(S^2)+\tr(S^3)\\
\dots&
\end{aligned}
$$
showing that, if $\tr(S^i)$ is constant for all $i\leq k-1$, then also $E_1,\dots,E_{k-1}$ are constant.

\smallskip

{\it Proof of (c)}. It simply follows from the fact that, if $\tilde H_f(t)=o(t^{\frac n2})$ for all $f\in C^{\infty}_0(\bd\Omega)$, then taking $k=n$ in the previous statement (b) we see that all mean curvatures $E_1,\dots,E_{n-1}$ are constant, hence the characteristic polynomial of $S$ is constant on $\bd\Omega$, which implies that $\bd\Omega$ is isoparametric by Theorem \ref{cartan}. The proof is now complete.


\subsection{Proof of Proposition \ref{level}}\label{plevel}

The proposition could be proved by writing the operator $A$ in  local normal coordinates $(r,y)$ around a boundary point $y\in\bd\Omega$. However, we give here a direct argument. 
Before giving the proof,   we introduce the following terminology. 

\nero {\it We say that  $\phi\in C^{\infty}(U)$ has {\rm level  $k$}  if  $k$ is the largest integer (including possibly $k=+\infty$)  such that
 $\phi,\derive{\phi}{\nu},\dots,\deriven{k}{\phi}{\nu}$ restrict to constant functions on $\bd\Omega$.}
 
 \smallskip
 
 By convention, if $\phi|_{\bd\Omega}$ is non constant then we say that $\phi$ has level $-\infty$. Note that if  $\phi|_{\bd\Omega}$ is constant then $\phi$ has level $k\geq 0$,  while  radial functions on $U$ have level $+\infty$.  In fact, if $\phi=f\circ\rho$ for a smooth function $f:[0,\epsilon)\to \reals$, then $\deriven k{\phi}{\nu}=f^{(k)}\circ\rho$ which, restricted to ${\bd\Omega}$, takes the constant value $f^{(k)}(0)$. 
 
\smallskip

By using Taylor expansion in the normal direction based at a generic point of the boundary, we easily obtain the following characterization.

\begin{lemme}\label{levelk} A function $\phi\in C^{\infty}(U)$ has level at least $k$ if and only if there exist smooth functions $f:[0,\epsilon)\to \reals$ and $\psi\in C^{\infty}(U)$ such that, on $U$:
$$
\phi=f\circ\rho+\rho^{k+1}\psi.
$$
In that case $f(r)$ will be a polynomial of degree $k$. 
\end{lemme}

We will also need the following fact. 
\begin{lemme}\label{obvious} Assume that $\eta$ has level at least $k$. If $\phi\in C^{\infty}(U)$ has level $h$, with $h\leq k$, then:

\item a) $N\phi$ has level at least $h-1$,

\item b) $\Delta_R\phi$ has level at least $h-2$,

\item c) $\Delta_T\phi$ has level at least $h$.

\smallskip

In particular, if $T$ is a polynomial of degree $m\leq h$ in the operators $N,\Delta_R,\Delta_T$  then $T\phi$ has level at least $h-m$.
\end{lemme}
\begin{proof} The assumptions imply existence of smooth functions $f,g$ on $[0,\epsilon)$ and $\psi,\xi$ on $U$ such that
$$
\twosystem
{\phi=f\circ\rho+\rho^{h+1}\psi}
{\eta=g\circ\rho+\rho^{k+1}\xi}
$$
One has:
$$
\derive{\phi}{\nu}=f'\circ\rho+\rho^h\Big((h+1)\psi+\rho\derive{\psi}{\nu}\Big),
$$
and by the previous lemma $\derive{\phi}{\nu}$ has level at least $h-1$. Since $h\leq k$ one sees immediately that $\eta\phi$ has level at least $h$, hence $N\phi$ has level at least $h-1$, showing 1). Part 2) is proved similarly; as for part 3) just observe that, as radial functions are constant on the level surfaces of $\rho$, one has
$
\Delta_T\phi=\rho^{h+1}\Delta_T\psi
$
showing that $\Delta_T\phi$ has level at least $h$.

The last assertion is now clear, as an obvious induction shows. 
\end{proof}

We now can prove Proposition \ref{level}, part (a). Given an operator $A\in{\cal A}(N,\Delta)$ of homogeneous degree $k+1$, recall the splitting $\Delta=\Delta_R+\Delta_T$. Accordingly, we can split 
$$
A=A_R+A_T,
$$
where $A_R$ is a polynomial in $N$ and $\Delta_R$ and $A_T$ contains at least one $\Delta_T$. For example, if $A=N\Delta^2$ then $A_R=N\Delta_R^2$ and $A_T=N\Delta_R\Delta_T+N\Delta_T\Delta_R+N\Delta_T^2$. 

Now $A_R$ derives 
functions only radially, and has degree $k+1$: then, the highest order of the normal derivative occurring in $A\eta$ is $k+1$.
By definition of weight  one sees easily that, on $U$, one has:
$$
A_R\eta=w(A)\deriven{k+1}{\eta}{\nu}+P\Big(\eta,\derive{\eta}{\nu}, \dots,\deriven{k}{\eta}{\nu}\Big),
$$
where the second term in the right is a polynomial in the functions $\eta,\derive{\eta}{\nu}, \dots,\deriven{k}{\eta}{\nu}$. By our  assumption,  all these functions restrict to constants  on $\bd\Omega$. By taking the restriction to the boundary we see
$$
A_R\eta|_{\bd\Omega}=w(A)\deriven{k+1}{\eta}{\nu}|_{\bd\Omega}+c_k,
$$
where $c_k$ is constant. As $A\eta=A_R\eta+A_T\eta$, the final assertion  now follows by proving that
$$
A_T\eta|_{\bd\Omega}=0.
$$
In order to do that, first observe that $A_T$ is a sum of terms of type $B\Delta_TC$, where $B$ is a polynomial in $N$ and $\Delta_R$, and $C$ is a polynomial in $N,\Delta_R,\Delta_T$. If $B$ has degree $h$ (possibly $h=0$) then $C$ has degree $k-h-1$.
By assumption $\eta$ has level at least $k$, hence (by Lemma \ref{obvious}) $C\eta$ has level at least $k-(k-h-1)=h+1$  and,  by Lemma \ref{levelk}, there exist smooth functions $f$ on $[0,\epsilon)$ and $\psi$ on $U$ such that:
$$
C\eta=f\circ\rho+\rho^{h+2}\psi.
$$
Therefore
$
\Delta_TC\eta=\rho^{h+2}\Delta_T\psi
$
vanishes on the boundary together with all its normal derivatives up and including the order $h+1$. As $B$ has order $h$ one must have $B(\rho^{h+2}\Delta_T\psi)=0$ on the boundary, hence
$$
B\Delta_TC\eta|_{\bd\Omega}=0,
$$
as asserted. This proves part (a). 

For part (b), recall that $B_{k+3}=-\bar D_{k+3}\eta$. Hence it is enough to apply (a) to the operator $A=\bar D_{k+3}$, which is homogeneous of degree $k+1$.


\section{Theorem \ref{combinatorial}:  the two main steps and the proof}\label{prec}

At this point all results of the paper are proved, except Theorem \ref{combinatorial}. The proof of Theorem \ref{combinatorial} is combinatorial, and will take the rest of the paper. We have split the proof in several sections. 

First, recall the heat content operators $\{D_k\}$ as defined in Theorem \ref{savoone}, and recall the barred operators $\{\bar D_k\}$ defined by the rule \eqref{dec}:  $D_k=\bar D_kN+\tilde D_k\Delta$. If $w(\bar D_k)$ denotes the weight of $\bar D_k$, we need to prove the identities:
\begin{equation}\label{barweights}
\twosystem
{w(\bar D_{2n})=\frac{2}{4^n n!},}
{w(\bar D_{2n+1})=\dfrac{1}{\sqrt\pi}\cdot\dfrac{1}{2^{n-1}(2n+1)!!}}
\end{equation}
for all integers $n$. There are two main steps: in Proposition \ref{weightofd} we compute the weight of $D_k$, and in Proposition \ref{weightofalpha} we compute the weight of the barred operators $\bar\alpha_n$, $\bar Z_{n+1}$ defined in \eqref{za}. Assuming Propositions \ref{weightofd} and \ref{weightofalpha}  we prove Theorem \ref{combinatorial} later in this section.  The proof of Proposition \ref{weightofd} will then be given in Section \ref{pweightofd}, while the proof of Proposition \ref{weightofalpha} is more involved, and will be given  in the last sections of the paper (Sections 7,8 and the Appendix). 

\smallskip

As a first step we compute the weights of $\{D_k\}$, by a functorial argument. 
\begin{prop}\label{weightofd} One has, for all $n\geq 1$:
$$
\twosystem
{w(D_{2n})=\dfrac{1}{n!}}
{w(D_{2n+1})=\dfrac{1}{\Gamma (n+\frac 32)}=\dfrac 1{\sqrt\pi}\cdot\dfrac{2^{n+1}}{(2n+1)!!}}
$$
\end{prop}

Proof: Section \ref{pweightofd}. 

Then, we compute the weights of the barred operators $\bar\alpha_n$ and $\bar Z_{n+1}$ (this is perhaps the main combinatorial difficulty). 

\begin{prop}\label{weightofalpha} One has:
$$
\threesystem
{w(\bar\alpha_0)=\frac 12,}
{w(\bar\alpha_n)=-\frac{3}{2^{n+1}(2n-1)!!}\quad\text{for all}\quad n\geq 1,}
{w(\bar Z_{n+1})=-\frac{1}{2^{n-1}(2n+1)!!}\quad\text{for all}\quad n\geq 1.}
$$
\end{prop}

Proof: Sections \ref{pwa}, \ref{ptu} and the Appendix. 


\subsection{Proof of Theorem \ref{combinatorial}} \label{proofofc}

Taking the bar on both sides of the recurrence scheme \eqref{recurrence} we see, thanks to rule $\overline{AB}=A\bar B$, that $\bar D_1=0$ and, for all $n\geq 1$:
\begin{equation}\label{barrecurrence}
\twosystem
{\bar D_{2n}=\dfrac{1}{\sqrt\pi}\sum_{i=1}^n\dfrac{\Gamma(i+\frac 12)\Gamma(n-i+\frac 12)}{n!}D_{2i-1}\bar\alpha_{n-i}}
{\bar D_{2n+1}=\dfrac{1}{\sqrt\pi}\bar Z_{n+1}+\dfrac{1}{\sqrt\pi}\sum_{i=1}^n\frac{i!\Gamma(n-i+\frac 12)}{\Gamma(n+\frac32)}D_{2i}\bar\alpha_{n-i}}
\end{equation}

We first prove that
$$
w(\bar D_{2n})=\dfrac{2}{4^n n!}.
$$
 First, recall the recurrence law for the Gamma function:
\begin{equation}\label{gamma}
\Gamma(\frac 12)=\sqrt\pi, \quad \Gamma(1+x)=x\Gamma (x), \quad\text{so that}\quad \Gamma(n+\frac 12)=\sqrt\pi\cdot\dfrac{(2n-1)!!}{2^n}.
\end{equation}
From Proposition \ref{weightofalpha}, $w(\bar\alpha_n)$ can be re-written:
\begin{equation}\label{cone}
w(\bar\alpha_n)=-\dfrac{3\sqrt\pi}{2^{2n+1}\Gamma(n+\frac 12)}.
\end{equation}
By taking weights in the first relation of \eqref{barrecurrence} we obtain:
\begin{equation}\label{contwo}
\begin{aligned}
\sqrt\pi n!w(\bar D_{2n})&=\sum_{i=1}^n\Gamma(i+\frac 12)\Gamma(n-i+\frac 12)w(D_{2i-1})
w(\bar\alpha_{n-i})\\
&=\sum_{i=1}^{n-1}\Gamma(i+\frac 12)\Gamma(n-i+\frac 12)w(D_{2i-1})
w(\bar\alpha_{n-i})+\Gamma(n+\frac 12)\Gamma(\frac 12)w(D_{2n-1})
w(\bar\alpha_{0})
\end{aligned}
\end{equation}
By  Proposition \ref{weightofd} we get
$
w(D_{2i-1})=\dfrac{1}{\Gamma(i+\frac 12)}
$
and from \eqref{cone} we obtain for $i\leq n-1$:
$$
w(\bar\alpha_{n-i})=-\dfrac{3\sqrt\pi}{2^{2(n-i)+1}\Gamma(n-i+\frac 12)}=
-\dfrac{\sqrt\pi\cdot 3\cdot 4^i}{2^{2n+1}\Gamma(n-i+\frac 12)}.
$$
Hence from \eqref{contwo}
$$
\sqrt\pi n!w(\bar D_{2n})=-\dfrac{3\sqrt\pi}{2^{2n+1}}\sum_{i=1}^{n-1}4^i+\dfrac{\sqrt\pi}{2}.
$$
As
$\sum_{i=1}^{n-1}4^i=\dfrac{4^n-4}{3}$ we get
$
w(\bar D_{2n})=\dfrac{2}{4^n n!}
$
which is the first identity  in  Theorem \ref{combinatorial}. 

The second assertion is proved similarly, by taking weights in the second relation of \eqref{barrecurrence}. We omit the details, which are straightforward.


\section{Theorem \ref{combinatorial}, step 1: proof of Proposition \ref{weightofd}}\label{pweightofd}

We start from the following lemma. 

\begin{lemme}\label{downtwo} Let $\Omega$ be any domain, let $\phi\in C^{\infty}(\Omega)$ and assume that $\phi=0$ on $\bd\Omega$. 
Let $\{D_k\}$ be the sequence of heat content operators. Then, for all $j\geq 1$:
\begin{equation}\label{downtwo}
\int_{\bd\Omega}D_{j+2}\phi=-\frac{2}{j+2}\int_{\bd\Omega}D_j\Delta\phi.
\end{equation}
\end{lemme}

\begin{proof} Recall that $\phi_t$ is the solution of the heat equation on $\Omega$ with initial data $\phi$ and Dirichlet boundary conditions. One knows from Section \ref{ahf}  that, as $t\to 0$, the heat content has  an asymptotic series
$$
\int_{\Omega}\phi_t\sim \int_{\Omega}\phi-\sum_{k=1}^{\infty}\beta_k(\phi)t^{\frac k2}.
$$
Differentiating both sides with respect to $t$ we obtain:
\begin{equation}\label{asone}
\frac{d}{dt}\int_{\Omega}\phi_t\sim-\sum_{k=1}^{\infty}\frac k2\beta_k(\phi)t^{\frac {k-2}2}\sim
-\sum_{j=-1}^{\infty}\frac {j+2}2\beta_{j+2}(\phi)t^{\frac {j}2}
\end{equation}
On the other hand we also know that $\int_{\Omega}\phi_t=\int_{\Omega}\phi u_t$ (see \eqref{phiut}). Therefore, as $\phi$ and $u_t$ vanish on the boundary:
\begin{equation}\label{astwo}
\begin{aligned}
\frac{d}{dt}\int_{\Omega}\phi_t &=\frac{d}{dt}\int_{\Omega}\phi u_t \\
&=-\int_{\Omega}\phi\Delta u_t\\
&=-\int_{\Omega}\Delta\phi \cdot u_t+\int_{\bd\Omega}u_t\derive{\phi}{\nu}-\int_{\bd\Omega}\derive{u_t}{\nu}\phi\\
&=-\int_{\Omega}\Delta\phi \cdot u_t\\
&=-\int_{\Omega}(\Delta\phi)_t\\
&\sim -\int_{\Omega}\Delta\phi +\sum_{j=1}^{\infty}\beta_j(\Delta\phi)t^{\frac {j}2}
\end{aligned}
\end{equation}
We now equate the  asymptotic series in \eqref{asone} and \eqref{astwo}: as $\phi$ vanishes on the boundary, one easily verifies that $\beta_1(\phi)=0, \beta_2(\phi)=\int_{\bd\Omega}\bd{\phi}/\bd{\nu}=\int_{\Omega}\Delta\phi$ and, for $j\geq 1$:
$$
-\dfrac{j+2}{2}\beta_{j+2}(\phi)=\beta_j(\Delta\phi).
$$
Recalling from Theorem \ref{savoone} that $\beta_j(\phi)=\int_{\bd\Omega}D_j\phi$ we get the assertion.
\end{proof} 

We can now prove the proposition. The crucial observation is that the coefficients of $D_k$ are universal, and do not depend on the domain $\Omega$: to compute the weight of $D_k$ we can just work on $\Omega=[0,1]$. Note that then $\bd\Omega=\{0,1\}$, and that the radial  vector field $\nu$, restricted to the neighborhood $U=\Omega\setminus\{\frac 12\}$ of $\bd\Omega$ where $\rho$ is regular,  is given by:
$$
\nu=\twosystem{\dfrac{d}{dx}\quad\text{on}\quad [0,\frac12)}
{-\dfrac{d}{dx}\quad\text{on}\quad (\frac12,1].}
$$
As the distance function $\rho$ is linear on $U$, we have $\eta=\Delta\rho=0$, hence 
$$
N\phi=2\derive{\phi}{\nu}, \quad \Delta\phi=-\deriven{2}{\phi}{\nu}
$$
and one immediately sees that, if $A\in{\cal A}(N,\Delta)$ is homogeneous of degree $k$, then:
$$
A\phi=w(A)\deriven{k}{\phi}{\nu}.
$$
Now, as $D_{j+2}$ is homogeneous of degree $j+1$:
\begin{equation}\label{dtwon}
\int_{\bd\Omega}D_{j+2}\phi=w(D_{j+2})\int_{\bd\Omega}\deriven{j+1}{\phi}{\nu}.
\end{equation}
On the other hand,  assuming that $\phi=0$ on $\bd\Omega$, we have by the lemma:
\begin{equation}\label{dtagain}
\begin{aligned}
\int_{\bd\Omega}D_{j+2}\phi &=-\dfrac{2}{j+2}\int_{\bd\Omega}D_{j}\Delta\phi\\
&=-\dfrac{2 }{j+2}w(D_j)\int_{\bd\Omega}\deriven{j-1}{\Delta\phi}{\nu}\\
&=\dfrac{2 }{j+2}w(D_j)\int_{\bd\Omega}\deriven{j+1}{\phi}{\nu}
\end{aligned}
\end{equation}
By \eqref{dtwon} and \eqref{dtagain} we conclude  that, for all $\phi\in C^{\infty}(\Omega)$ which  vanish on the boundary one has:
$$
\Big(w(D_{j+2})-\frac 2{j+2}w(D_{j})\Big)\int_{\bd\Omega}\deriven{j+1}{\phi}{\nu}=0.
$$
We can always choose $\phi$ so that the integral on the right does not vanish; hence, for all $j$:
$$
w(D_{j+2})=\frac 2{j+2}w(D_{j}).
$$
 As $w(D_2)=1$ (see table \eqref{table}) we obtain $w(D_{2n})=\frac{1}{n!}$ as asserted. Similarly, since  $w(D_1)=\frac{2}{\sqrt\pi}$ we obtain
$w(D_{2n+1})=\frac{1}{\Gamma(n+\frac32)}.$ The proof is complete.


\section{Theorem \ref{combinatorial}, step 2: proof of  Proposition \ref{weightofalpha}}\label{pwa}

Perhaps, this is the step which is more involved from a combinatorial point of view. 
We first make a change of variables in the original definition of the operators $R_{nj},S_{nj}$ in  \eqref{rsone} to simplify the subsequent calculations. 
Then, the proof of Proposition \ref{weightofalpha} will reduce to the proof  of Proposition \ref{tu} below, which in turn will be given in the last section. 

\smallskip

For non-negative integers $n$ and $j$, let us introduce new operators $P_{nj},Q_{nj}\in {\cal A}(N,\Delta)$ as follows:
$$
\twosystem
{P_{nj}=-\frac{2}{4^j}R_{n+j,j},}
{Q_{nj}=-\frac{4}{4^j}S_{n+j,j}.}
$$
Then, one easily shows that the defining  scheme \eqref{rsone} takes the form:
\begin{equation}\label{recurrencepq}
\threesystem
{P_{nj}=-(N^2+\Delta)P_{n-1,j}+\frac12NQ_{n-1,j}}
{Q_{nj}=2\Delta NP_{n-1,j}-\Delta Q_{n-1,j}+\frac 12NP_{n,j-1}}
{P_{00}=-2I,\quad Q_{00}=0.}
\end{equation}
with the understanding that $P_{nj}$ and $Q_{nj}$ are zero whenever $n$ or $j$ is negative. 
Introduce the sequence
\begin{equation}\label{eien}
a_n=\binom{2n-1}{n}= \dfrac{2^{n-1}(2n-1)!!}{n!}\quad\text{for $n\geq 1$}.
\end{equation}
The first few terms are 
$
1,3,10,35,126,462,1716,6735,\dots
$
Recalling the relations \eqref{gamma} of the Gamma function  it is straightforward to verify that the operators $Z_{n+1}$ and $\alpha_n$ of \eqref{za} can be written, for $n\geq 1$:
\begin{equation}\label{newzalpha}
\twosystem
{Z_{n+1}=-\dfrac{1}{2^{n-1}(2n+1)!!}\sum_{j=0}^na_{n+j}P_{nj}}
{\alpha_n=-\dfrac{1}{2^{n+1}(2n-1)!!}\sum_{j=0}^{n+1}a_{n+j}Q_{nj}.}
\end{equation}
For $n=0$, it can be directly checked from the original recursive scheme \eqref{za} that 
$$
\alpha_0=\frac 12 N
$$ 
(in fact $\{0,1\}=\frac 12, S_{00}=0, S_{11}=N$). Then $\bar\alpha_0=\frac 12I$ and $w(\bar\alpha_0)=\frac 12$.

\smallskip

Then,  we need to verify the value of  $w(\bar\alpha_n)$ and $w(\bar Z_{n+1})$ given in Proposition \ref{weightofalpha} only when $n\geq 1$. By taking bars in \eqref{newzalpha} and then taking weights we see
\begin{equation}\label{barzalpha}
\twosystem
{w(\bar Z_{n+1})=-\dfrac{1}{2^{n-1}(2n+1)!!}\sum_{j=0}^na_{n+j}w(\bar P_{nj})}
{w(\bar\alpha_n)=-\dfrac{1}{2^{n+1}(2n-1)!!}\sum_{j=0}^{n+1}a_{n+j}w(\bar Q_{nj}).}
\end{equation}
From the scheme \eqref{recurrencepq} we see that
$$
\twosystem{P_{01}=0}{P_{10}=2(N^2+\Delta)}, \quad \twosystem{Q_{01}=-N}{Q_{10}=-4\Delta N}
$$
so that
$$
\twosystem{\bar P_{01}=0}{\bar P_{10}=2N}, \quad \twosystem{\bar Q_{01}=-I}{\bar Q_{10}=-4\Delta}.
$$
Taking the bar in the scheme \eqref{recurrencepq} we  see that $\bar P_{nj}$ and $\bar Q_{nj}$ are uniquely determined by the scheme:
\begin{equation}\label{recurrencebar}
\threesystem
{\bar P_{nj}=-(N^2+\Delta)\bar P_{n-1,j}+\frac12N\bar Q_{n-1,j}}
{\bar Q_{nj}=2\Delta N\bar P_{n-1,j}-\Delta \bar Q_{n-1,j}+\frac 12N\bar P_{n,j-1}}
{\bar P_{01}=0,\quad \bar P_{10}=2N,\quad \bar Q_{01}=-I, \quad \bar Q_{10}=-4\Delta.}
\end{equation}
For non-negative integers $n$ and $j$ we now define real numbers (actually integers) by:
$$
T_{nj}\doteq w(\bar P_{nj}), \quad U_{nj}\doteq w(\bar Q_{nj}).
$$
Recalling that taking  weights is an algebra homomorphism, we see from \eqref{recurrencebar} that $T_{nj}$ and $U_{nj}$ must satisfy the recurrence scheme:
\begin{equation}\label{recurrencetu}
\threesystem
{T_{nj}=-3T_{n-1,j}+U_{n-1,j}}
{U_{nj}=-4T_{n-1,j}+U_{n-1,j}+T_{n,j-1}}
{\twovector{T_{01}}{U_{01}}=\twovector 0{-1},\quad \twovector{T_{10}}{U_{10}}=\twovector 44}
\end{equation}
It is now clear that Proposition \ref{weightofalpha} is implied by \eqref{barzalpha} and the following result.

\begin{prop}\label{tu} Let $T_{nj}$ and $U_{nj}$ be the integers defined by the scheme \eqref{recurrencetu} and let $\{a_n\}$ be the sequence defined in \eqref{eien}. Then, for all $n\geq 1$ one has:
$$
\sum_{j=0}^na_{n+j}T_{nj}=1, \quad \sum_{j=0}^{n+1}a_{n+j}U_{nj}=3.
$$
\end{prop}


\section{Last step: proof of Proposition \ref{tu}}\label{ptu}

The proof depends on a  combinatorial result (see Lemma \ref{mainlemma}), which can be proved in the framework of orthogonal polynomials and Hankel transforms. For the proof, we refer to the Appendix (Section \ref{appendix}) where we will use the results  in  \cite{Tamm} (see also \cite {C-M} for related work with the sequence of Catalan numbers). 

For the sequence 
\begin{equation}\label{an}
\{a_n\}=\Big\{\binom{2n-1}{n}\Big\}_{n\geq 1}=\{1,3,10,35,126,462,1716,6735,\dots\}
\end{equation} 
defined in \eqref{eien}  introduce the infinite matrices:
$$
A=\matrice\quattro 1 4 {4^2}{\cdots}\quattro{a_1}{a_2}{a_3}{\cdots}\quattro{a_2}{a_3}{a_4}{\cdots}
\quattro{\vdots}{\vdots}{\vdots}{\ddots}\ok,\quad
A'=\matrice\quattro 1 4 {4^2}{\cdots}\quattro{a_2}{a_3}{a_4}{\cdots}\quattro{a_3}{a_4}{a_5}{\cdots}
\quattro{\vdots}{\vdots}{\vdots}{\ddots}\ok.
$$
Note that the first row involves the powers of $4$ and that $A'$ corresponds to the shifted sequence $\{a_2,a_3,\dots\}$. Let  $A_n$ (resp. $ A'_n$) be the square matrix of order $n+1$ in the upper left-hand corner of $A$ (resp. $A'$) so that, for example:
$$
A_1=\twomatrix 14{a_1}{a_2}, \quad A_2=\threematrix 14{4^2}{a_1}{a_2}{a_3}{a_2}{a_3}{a_4}, \quad \dots
$$ 
 The following identities will be proved in the Appendix (Section \ref{appendix}).

\begin{lemme}\label{mainlemma} For all $n\geq 1$ one has:
$$
\det A_n=(-1)^{n} \quad\text{and}\quad \det A'_n=(-1)^{n}(n+1).
$$
\end{lemme}

We also need the following  lemma, which follows easily by induction using the recursive scheme defining the numbers $T_{nj}$ and $U_{nj}$ in \eqref{recurrencetu}, and for which we omit the proof. 

\begin{lemme}\label{omitproof} For all $n\geq 1$ one has:
$$
T_{nn}=U_{n,n+1}=-1, \quad T_{n0}=(-1)^{n+1}4n, \quad U_{n0}=(-1)^{n+1}4(2n-1).
$$
Moreover $T_{nj}\ne 0$ only when $0\leq j\leq n$ and $U_{nj}\ne 0$ only when $0\leq j\leq n+1$.
\end{lemme}

We will prove, by induction on $n$, the following set of identities. Note that identities $(c.1)$ and $(c.2)$ are precisely the statement we want to show (that is, Proposition \ref{tu}). 
$$
\threesystem
{\sum_{j=0}^{n}4^jT_{nj}=\sum_{j=0}^{n+1}4^jU_{nj}=0 \quad & (a.1), (a.2)}
{\sum_{j=0}^{n}a_{k+j}T_{nj}=\sum_{j=0}^{n+1}a_{k+j}U_{nj}=0\quad\text{for all $k=1,\dots,n-1$,}
\quad & (b.1), (b.2)}
{\sum_{j=0}^{n}a_{n+j}T_{nj}=1,\quad \sum_{j=0}^{n+1}a_{n+j}U_{nj}=3. \quad &(c.1), (c.2)}
$$

Let us introduce a useful notation. We let $B_n$ be the $n\times (n+1)$ matrix obtained by deleting the last row from $A_n$, and $B_n^{(j)}$ the $n\times n$ matrix obtained by deleting the $(j+1)$-th column from $B_n$ (that is, the column containing $4^j$). Then, in particular:
$$
B_n^{(n)}=A_{n-1}.
$$
For example:
$$
A_2=\threematrix 14{4^2}{a_1}{a_2}{a_3}{a_2}{a_3}{a_4},\, 
B_2=\matrice \tre 14{4^2}\tre {a_1}{a_2}{a_3}\ok, 
$$
and
$$
B_2^{(0)}=\twomatrix 4{4^2}{a_2}{a_3}, \quad B_2^{(1)}=\twomatrix 1{4^2}{a_1}{a_3}, \quad B_2^{(2)}=\twomatrix 14{a_1}{a_2}=A_1.
$$
Laplace rule applied to the last row of $A_n$ gives, since $\det A_n=(-1)^n$:
\begin{equation}\label{lr}
\sum_{j=0}^{n}(-1)^ja_{n+j}\det B_n^{(j)}=1.
\end{equation}

Next, define the following vectors (in $\real{n+1}$ and $\real{n+2}$ respectively):
$$
\vec T_n=\fourvector{T_{n0}}{T_{n1}}{\vdots}{T_{nn}},
\quad \vec U_n=\fourvector{U_{n0}}{U_{n1}}{\vdots}{U_{n,n+1}}.
$$
It is clear that the identities (a.1), \dots, (c.2) are equivalent in vectorial form to:
\begin{equation}\label{vectorial}
A_n\vec T_n=\fourvector {0}{\vdots}{0}{1}, \quad B_{n+1}\vec U_n=\fourvector {0}{\vdots}{0}{3}.
\end{equation}

Then, let us prove \eqref{vectorial} by induction on $n$. We first  verify it for $n=1$. From the recursive scheme \eqref{recurrencetu} we see that
$\vec T_1=\twovector{4}{-1}$ and $ \vec U_1=\threevector 43{-1}$.
Hence:
$$
A_1\vec T_1=\twomatrix 1413\twovector 4{-1}=\twovector 01, 
\quad B_2\vec U_1=\matrice \tre 14{16}\tre 13{10}\ok\threevector 43{-1}=\twovector 0{3}.
$$
which shows the assertion. As a double-check, we verify it also for $n=2$:
$$
A_2\vec T_2=\matrice \tre 14{16}\tre 13{10}\tre 3{10}{35}\ok\threevector {-8}{6}{-1}=\threevector 001,
\quad B_3\vec U_2=\matrice\quattro 14{16}{64}\quattro 13{10}{35}\quattro 3{10}{35}{126}\ok
\fourvector {-12}{-1}{5}{-1}=\threevector 003.
$$
We now assume that \eqref{vectorial} is true for $n$ (equivalently, we assume the identities (a.1),\dots, (c.2) above)  and prove it for $n+1$, that is, we must prove that
$$
\threesystem
{\sum_{j=0}^{n+1}4^jT_{n+1,j}=\sum_{j=0}^{n+2}4^jU_{n+1,j}=0 \quad & (1.1), (1.2)}
{\sum_{j=0}^{n+1}a_{k+j}T_{n+1,j}=\sum_{j=0}^{n+2}a_{k+j}U_{n+1,j}=0\quad\text{for all $k=1,\dots,n$}\quad & (2.1), (2.2)}
{\sum_{j=0}^{n+1}a_{n+1+j}T_{n+1,j}=1,\quad \sum_{j=0}^{n+2}a_{n+1+j}U_{n+1,j}=3 \quad &(3.1), (3.2)}
$$
We prove (1.1), (2.1), (3.1), (1.2), (2.2), (3.2) in that order. We recall from Lemma \ref{omitproof} that
$T_{nj}\ne 0$ only when $0\leq j\leq n$ and $U_{nj}\ne 0$ only when $0\leq j\leq n+1$.

\medskip

{\bf Proof of (1.1).} It follows from (a.1), (a.2) and the first equation in the scheme \eqref{recurrencetu}:
$T_{n+1,j}=-3T_{nj}+U_{nj}$.

\medskip

{\bf Proof of (2.1).} For $k=1,\dots, n-1$ follows from the first equation in the scheme \eqref{recurrencetu}, (b.1), (b.2). For $k=n$ it follows from (c.1), (c.2).

\medskip

{\bf Proof of (3.1).} From (1.1) and (2.1) we see that
$
B_{n+1}\vec T_{n+1}=0.
$
The vector $\vec T_{n+1}$ is then a  solution of a homogeneous system of $n+1$ equations in $n+2$ unknowns with coefficient matrix $B_{n+1}$. Therefore, there exists $\lambda\in\reals$ such that
$$
T_{n+1,j}=(-1)^j\lambda \det B_{n+1}^{(j)}
$$
for all $j=0,\dots,n+1$. We know from Lemma \ref{omitproof}  that $T_{n+1,n+1}=-1$ and from Lemma \ref{mainlemma} that  $\det{A_n}=(-1)^n$. Therefore:
$$
\begin{aligned}
-1&= T_{n+1,n+1}\\
&=(-1)^{n+1}\lambda \det{B_{n+1}^{(n+1)}}\\
&=(-1)^{n+1}\lambda \det{A_n}\\
&=(-1)^{n+1}(-1)^n\lambda\\
&=-\lambda
\end{aligned}
$$
hence $\lambda=1$ and, for all $j$:
$$
T_{n+1,j}=(-1)^j \det{B_{n+1}^{(j)}}.
$$
On the other hand we know from \eqref{lr} that for all $k$ one has $\sum_{j=0}^{k}(-1)^ja_{k+j}\det{B_k^{(j)}}=1$
hence, taking $k=n+1$:
$$
\sum_{j=0}^{n+1}a_{n+1+j}T_{n+1,j}=\sum_{j=0}^{n+1}(-1)^ja_{n+1+j}\det{B_{n+1}^{(j)}}=1
$$
and (3.1) follows.

\smallskip

{\bf Proof of (1.2).} One has from the second relation in  scheme \eqref{recurrencetu}:
$$
\sum_{j=0}^{n+2}4^jU_{n+1,j}=-4\sum_{j=0}^{n+2}4^jT_{nj}+\sum_{j=0}^{n+2}4^jU_{nj}+
\sum_{j=0}^{n+2}4^jT_{n+1,j-1}
$$
The first two terms are zero by (a.1) and (a.2), the third equals $4\sum_{j=0}^{n+2}4^{j-1}T_{n+1,j-1}
=4\sum_{i=0}^{n+1}4^iT_{n+1,i}
=0$ by (1.1).

\smallskip

{\bf Proof of (2.2)} One has from the second relation in scheme \eqref{recurrencetu}:
$$
\begin{aligned}
\sum_{j=0}^{n+2}a_{k+j}U_{n+1,j}&=-4\sum_{j=0}^{n+2}a_{k+j}T_{nj}+\sum_{j=0}^{n+2}a_{k+j}U_{nj}+\sum_{j=0}^{n+2}a_{k+j}T_{n+1,j-1}
\end{aligned}
$$
If $k\leq n-1$ the first two terms on the right are zero by (b.1) and (b.2), and then:
$$
\sum_{j=0}^{n+2}a_{k+j}U_{n+1,j}=\sum_{j=0}^{n+2}a_{k+j}T_{n+1,j-1}=\sum_{i=0}^{n+1}a_{k+i+1}T_{n+1,i}=0
$$
because of (2.1). If $k=n$ the right-hand side equals 
$$
\begin{aligned}
&-4\sum_{j=0}^{n+2}a_{n+j}T_{nj}+\sum_{j=0}^{n+2}a_{n+j}U_{nj}+\sum_{j=0}^{n+2}a_{n+j}T_{n+1,j-1}\\
&=-4\sum_{j=0}^na_{n+j}T_{nj}+\sum_{j=0}^{n+1}a_{n+j}U_{nj}+\sum_{i=0}^{n+1}a_{n+1+i}T_{n+1,i}
\\
&=-4+3+1\\
&=0
\end{aligned}
$$
by (c.1), (c.2) and (3.1).

\medskip

{\bf Proof of (3.2)} Set
\begin{equation}\label{set}
\sum_{j=0}^{n+2}a_{n+1+j}U_{n+1,j}=\lambda.
\end{equation}
We have to show that $\lambda=3$. As
$
T_{n+2,j}=-3T_{n+1,j}+U_{n+1,j}
$
we see that, from (1.1) and (1.2), we have $\sum_j4^jT_{n+2,j}=0$ and
$$
\sum_{j=0}^{n+2}a_{k+j}T_{n+2,j}=-3\sum_{j=0}^{n+1}a_{k+j}T_{n+1,j}+\sum_{j=0}^{n+2}a_{k+j}U_{n+1,j}.
$$
If $k\leq n$ this is zero by (2.1) and (2.2). When $k=n+1$ this is $\lambda-3$ by (3.1) and \eqref{set}. Expressed in matrix form, all this becomes the statement:
\begin{equation}\label{statement}
B_{n+2}\vec T_{n+2}=\fourvector {0}{\vdots}{0}{\lambda-3}.
\end{equation}
As $T_{n+2,n+2}=-1$ and $B_{n+2}^{(n+2)}=A_{n+1}$, taking the terms involving $T_{n+2,n+2}$  to the right-hand side we see that \eqref{statement} can be written:
$$
A_{n+1}\fourvector {T_{n+2,0}}{T_{n+2,1}}{\vdots}{T_{n+2,n+1}}=\fourvector {4^{n+2}}{a_{n+3}}{\vdots}{a_{2n+3}+\lambda-3}
$$
By Cramer's rule, $T_{n+2,0}$ is the ratio:
\begin{equation}\label{cramer}
T_{n+2,0}=\dfrac{\det{C_{n+2}}}{\det{A_{n+1}}},
\end{equation}
where $C_{n+2}$ is the $(n+2)\times(n+2)$ matrix obtained replacing the first column of $A_{n+1}$ by the column 
$$
\fourvector {4^{n+2}}{a_{n+3}}{\vdots}{a_{2n+3}+\lambda-3}=\fourvector {4^{n+2}}{a_{n+3}}{\vdots}{a_{2n+3}}+\fourvector {0}{0}{\vdots}{\lambda-3}.
$$
Accordingly:
$$
\det{C_{n+2}}=\det\matrice\quattro{4^{n+2}}{4}{\hdots}{4^{n+1}}
\quattro{a_{n+3}}{a_2}{\hdots}{a_{n+2}}
\quattro{\hdots}{\hdots}{\hdots}{\hdots}
\quattro{a_{2n+3}}{a_{n+2}}{\hdots}{a_{2n+2}}\ok
+
\det\matrice\quattro{0}{4}{\hdots}{4^{n+1}}
\quattro{0}{a_2}{\hdots}{a_{n+2}}
\quattro{\hdots}{\hdots}{\hdots}{\hdots}
\quattro{\lambda-3}{a_{n+2}}{\hdots}{a_{2n+2}}\ok.
$$
Reordering columns and taking the factor $4$ out of the determinants, we see:
$$
\det{C_{n+2}}=(-1)^{n+1}4\det{A'_{n+1}}+(-1)^{n+1}4(\lambda-3)\det{A'_n}.
$$
We know from Lemma \ref{mainlemma}  that
$$
\quad \det{A'_{n+1}}=(-1)^{n+1}(n+2), \quad \det{A'_n}=(-1)^n(n+1).
$$
hence
$$
\det{C_{n+2}}=4(n+2)-4(\lambda-3)(n+1).
$$
On the other hand, by \eqref{cramer} and Lemma \ref{omitproof}, as $T_{n+2,0}=(-1)^{n+3}4(n+2)$: 
$$
\begin{aligned}
\det{C_{n+2}}&=(\det{A_{n+1}})T_{n+2,0}\\
&=4(n+2)
\end{aligned}
$$
Comparing these  two last expressions we indeed get $\lambda-3=0$ hence $\lambda=3$. With this, the proof is complete.




\section{Appendix: proof of Lemma \ref{mainlemma}} \label{appendix}

This section is based on the paper by Ulrich Tamm \cite{Tamm}, and we will follow closely the notation there. 

Given a sequence $\{c_0,c_1,c_2,\dots\}$ form the infinite matrix 
$$
A=\matrice\quattro {c_0}{c_1}{c_2}{\cdots}\quattro{c_1}{c_2}{c_3}{\cdots}\quattro{c_2}{c_3}{c_4}{\cdots}
\quattro{\vdots}{\vdots}{\vdots}{\ddots}\ok
$$
and let $A_n$ be the $n\times n$ sub-matrix in the upper left corner:
$$
A_1=(c_0), \quad A_2=\twomatrix {c_0}{c_1}{c_1}{c_2}, \quad A_3=\threematrix {c_0}{c_1}{c_2}{c_1}{c_2}{c_3}{c_2}{c_3}{c_4}, \dots
$$

In this way we obtain a sequence of Hankel matrices. We will also consider shifted sequences,  namely for all $k\geq 0$ we set
$$
A^{(k)}=\matrice\quattro {c_k}{c_{k+1}}{c_{k+2}}{\cdots}\quattro{c_{k+1}}{c_{k+2}}{c_{k+3}}{\cdots}\quattro{c_{k+2}}{c_{k+3}}{c_{k+4}}{\cdots}
\quattro{\vdots}{\vdots}{\vdots}{\ddots}\ok
$$
and let $A^{(k)}_n$ be the $n\times n$ sub-matrix in the upper left corner of $A^{(k)}$;  by convention $A^{(0)}=A$. For all $k=0,1,\dots$ we set
$$
d_n^{(k)}=\det A^{(k)}_n.
$$
By definition, the sequence $\{d_1^{(k)}, d_2^{(k)}, d_3^{(k)}, \dots\}$ is called the {\it Hankel transform} of the sequence $\{c_k,c_{k+1}, c_{k+2}\dots\}$.

\smallskip

Our interest is in the sequence 
$$
c_m=\binom{2m+1}{m}, \quad m\geq 0.
$$
It is just the sequence defined in \eqref{eien} with offset at $m=0$:
$$
c_m=a_{m+1}=\{1,3,10,35,126,462,1716, 6735, \dots\}.
$$
Then
$$
A^{(0)}=\matrice\cinque 13{10}{35}{\cdots}\cinque 3{10}{35}{126}{\cdots}\cinque {10}{35}{126}{462}{\cdots}
\cinque{\vdots}{\vdots}{\vdots}{\vdots}{\ddots}\ok, \quad\text{and}\quad
A^{(1)}=\matrice\cinque 3{10}{35}{126}{\cdots}\cinque {10}{35}{126}{462}{\cdots}\cinque {35}{126}{462}{1716}{\cdots}
\cinque{\vdots}{\vdots}{\vdots}{\vdots}{\ddots}\ok.
$$
From part b) of Proposition 2.1 in \cite{Tamm} we see that, for all $n$:
\begin{equation}\label{dn}
d_n^{(0)}=1, \quad d_n^{(1)}=2n+1,
\end{equation}
and the following expression holds for $k\geq 2$:
$$
d_n^{(k)}=\Pi _{1\leq i\leq j\leq k}\dfrac{i+j-1+2n}{i+j-1}.
$$

Now introduce an indeterminate $x$ and consider the infinite matrices
$$
A=\matrice\quattro 1 x{x^2}{\cdots}\quattro {c_0}{c_1}{c_2}{\cdots}\quattro{c_1}{c_2}{c_3}{\cdots}\quattro{c_2}{c_3}{c_4}{\cdots}
\quattro{\vdots}{\vdots}{\vdots}{\ddots}\ok, \quad
A^{(1)}=\matrice\quattro 1 x{x^2}{\cdots}\quattro {c_1}{c_2}{c_3}{\cdots}\quattro{c_2}{c_3}{c_4}{\cdots}\quattro{c_3}{c_4}{c_5}{\cdots}
\quattro{\vdots}{\vdots}{\vdots}{\ddots}\ok
$$
and the sequences of polynomials $\{P_n(x)\}, \{P_n^{(1)}(x)\}$ defined by
$$
\twosystem
{P_0(x)=1, \quad P_1(x)=\twodet 1x{c_0}{c_1}, \quad P_2(x)=\threedet 1x {x^2}{c_0}{c_1}{c_2}{c_1}{c_2}{c_3}, \quad ...}
{P_0^{(1)}(x)=1, \quad P_1^{(1)}(x)=\twodet 1x{c_1}{c_2}, \quad P_2^{(1)}(x)=\threedet 1x {x^2}{c_1}{c_2}{c_3}{c_2}{c_3}{c_4}, \quad ...}
$$
Recalling that $c_m=a_{m+1}$ we see that proving Lemma \ref{mainlemma} amounts to prove that, for all $n$:
\begin{equation}\label{pol}
P_n(4)=(-1)^n, \quad P_n^{(1)}(4)=(-1)^n(n+1).
\end{equation}
Note that, by \eqref{dn},  the leading coefficient of $P_n(x)$ (resp. $P_n^{(1)}(x)$) is $(-1)^{n}d_n=(-1)^n$ 
(resp.  $(-1)^nd_n^{(1)}=(-1)^n(2n+1)$). Then, the polynomials
\begin{equation}\label{tn}
t_n(x)=(-1)^{n}P_n(x), \quad t_n^{(1)}(x)=\dfrac{(-1)^n}{2n+1}P_n^{(1)}(x)
\end{equation}
are of degree $n$ and monic. In combinatorics, the sequence $\{t_n(x)\}$ is called the {\it sequence of orthogonal polynomials associated to $\{c_0,c_1,\dots\}$} (they coincide with the polynomials defined in (1.8) in \cite{Tamm}).

Taking into account  \eqref{pol} and \eqref{tn}, to prove Lemma \ref{mainlemma}  it is then sufficient to show the following fact.

\begin{lemme}\label{polt} For the polynomials defined in \eqref{tn} one has, for all $n$:
$$
t_n(4)=1,\quad t_n^{(1)}(4)=\dfrac{n+1}{2n+1}.
$$
\end{lemme}

\begin{proof}
It turns out that $\{t_n(x)\}$ and $\{t_n^{(1)}(x)\}$ satisfy a three-term recursive relation:
\begin{equation}\label{rp}
\twosystem
{t_n(x)=(x-\alpha_n)t_{n-1}(x)-\beta_{n-1}t_{n-2}(x), \quad t_0(x)=1, t_1(x)=x-\alpha_1}
{t_n^{(1)}(x)=(x-\alpha_n^{(1)})t_{n-1}^{(1)}(x)-\beta_{n-1}^{(1)}t_{n-2}^{(1)}(x) \quad t_0^{(1)}(x)=1, 
t_1^{(1)}(x)=x-\alpha_1^{(1)}}
\end{equation}
for suitable numerical sequences $\alpha_n\doteq \alpha_n^{(0)},\beta_n\doteq \beta_n^{(0)}, \alpha_n^{(1)}, \beta_{n-1}^{(1)}$. This sequences have been computed in \cite{Tamm}. 
In fact, for arbitrary $k$, the coefficients  $\alpha_n^{(k)}, \beta_{n-1}^{(k)}$ can be computed in terms of the coefficients $q_n^{(k)},e_n^{(k)}$ in the continued fraction expansion of $1-xF(x)$, where 
$F(x)=\sum_{m=0}^{\infty}c_{m+k}x^m$ (see (1.14), (1.19)  and (1.20) in \cite{Tamm}):
\begin{equation}\label{qe}
\alpha_1=q_1\quad \text{and, for $n\geq 1$}: \quad  \alpha_{n+1}^{(k)}=q_{n+1}^{(k)}+e_{n}^{(k)},\quad \beta_n^{(k)}=q_n^{(k)}\cdot e_n^{(k)},
\end{equation}
For our sequence $c_m=\binom{2m+1}{m}$ the corresponding coefficients have been computed in Corollary 2.1, equation (2.6):
$$
q_n^{(k)}=\dfrac{(2n+2k)(2n+2k+1)}{(2n+k-1)(2n+k)},\quad
e_n^{(k)}=\dfrac{(2n-1)(2n)}{(2n+k)(2n+k+1)}.
$$
In particular,
$$
\twosystem
{q_n=q_n^{(0)}=\dfrac{2n+1}{2n-1}}
{e_n=e_n^{(0)}=\dfrac{2n-1}{2n+1}}
\quad\text{and}\quad 
\twosystem
{q_n^{(1)}=\dfrac{(2n+2)(2n+3)}{(2n)(2n+1)}}
{e_n^{(1)}=\dfrac{(2n-1)(2n)}{(2n+1)(2n+2)}.}
$$

From \eqref{qe}  we obtain $\alpha_1=q_1=3$, $\alpha_1^{(1)}= q_1^{(1)}=\frac{10}{3}$ and for $n\geq 2$ (after some calculations):
\begin{equation}\label{kzero}
\twosystem
{\alpha_n=2}
{\beta_{n-1}=1},
\quad\text{and}\quad
\twosystem
{\alpha_n^{(1)}=2+\frac{4}{(2n+1)(2n-1)}}
{\beta_{n-1}^{(1)}=\dfrac{(2n-3)(2n+1)}{(2n-1)^2}}
\end{equation}

From \eqref{rp} and \eqref{kzero}  one has the following recursive scheme for the sequence $\{t_n(x)\}$:
$$
\threesystem
{t_n(x)=(x-2)t_{n-1}(x)-t_{n-2}(x)}
{t_0(x)=1}
{t_1(x)=x-3}
$$
Setting $X_n=t_n(4)$ we obtain
$$
\threesystem
{X_n=2X_{n-1}-X_{n-2}}
{X_0=1}
{X_1=1}
$$
hence $X_n=1$ for all $n$. That is:
$$
t_n(4)=1
$$
for all $n$. This shows the first relation in the Lemma. 

\medskip

About the sequence $\{t_n^{(1)}(x)\}$ we know:
$$
\threesystem
{t_n^{(1)}(x)=(x-\alpha_n^{(1)})t_{n-1}^{(1)}(x)-\beta_{n-1}^{(1)}t_{n-2}^{(1)}(x)}
{t_0^{(1)}(x)=1}
{t_1^{(1)}(x)=x-\frac{10}3}
$$
where $\alpha_n^{(1)}$ and $\beta_{n-1}^{(1)}$ are as in \eqref{kzero}. 
We are interested in the sequence 
$
X_n=t_n^{(1)}(4),
$
which obeys the recursive law:
$$
\threesystem
{X_n=(4-\alpha_n^{(1)})X_{n-1}-\beta_{n-1}^{(1)}X_{n-2}}
{X_0=1}
{X_1=\frac 23}
$$
By induction on $n$, let us show that $X_n=\frac{n+1}{2n+1}$. In fact this holds for $n=1$. Assume it to be true for all $k\leq n-1$. By \eqref{kzero}:
$$
4-\alpha_n^{(1)}=\dfrac{8n^2-6}{(2n-1)(2n+1)}, \quad \beta_{n-1}^{(1)}=\dfrac{(2n-3)(2n+1)}{(2n-1)^2}
$$
hence:
$$
\begin{aligned}
X_n&=(4-\alpha_n^{(1)})X_{n-1}-\beta_{n-1}^{(1)}X_{n-2}\\
&=\dfrac{8n^2-6}{(2n-1)(2n+1)}\cdot\frac{n}{2n-1}-\dfrac{(2n-3)(2n+1)}{(2n-1)^2}\cdot\frac{n-1}{2n-3}\\
&=\frac{n+1}{2n+1}
\end{aligned}
$$
which shows the claim. With this, the proof of Lemma \ref{mainlemma} is complete.

\end{proof}

{\bf Remark.} We point out that the expression of $\alpha_{n+1}^{(k)}$ in the statement of Corollary 2.3 in \cite{Tamm} is wrong; this is due to an incorrect algebraic manipulation of $q_{n+1}^{(k)}+e_n^{(k)}$ in the proof of Corollary 2.3 there. The correct value of $\alpha_{n+1}^{(1)}$ shown in \eqref{kzero} is directly computed from the expressions of $q_n^{(1)}$ and $e_n^{(1)}$ taken from Corollary 2.1, equation (2.6) of \cite{Tamm}. 

\medskip

{\bf Acknowledgements.} We wish to thank Sylvestre Gallot for his interest and for many enlightening discussions,  and Stefano Capparelli for pointing out to us  the role of orthogonal polynomials  in the combinatorial part of the paper (Lemma \ref{mainlemma}).



\end{document}